\documentclass[a4paper,12pt,leqno]{article}
\usepackage[utf8]{inputenc}

\usepackage[top=3cm, left=2.5cm, right=2.5cm, bottom=3cm]{geometry}

\usepackage{hyperref}
\hypersetup{colorlinks=true,allcolors=dodger}

\hypersetup{
    colorlinks=true,
    linkcolor=dodger,
    filecolor=dodger,
    urlcolor=dodger,
    citecolor=dodger,
}

\usepackage{amsmath,mathtools}
\usepackage{amsthm, pb-diagram}
\usepackage{amssymb,comment}
\usepackage{amsfonts,graphicx,color}
\usepackage{enumitem, fancyhdr, dsfont}
\usepackage[normalem]{ulem}
\usepackage{thmtools,cleveref}
\usepackage{stmaryrd}
\usepackage{fontawesome5}
\usepackage{lineno,multicol}

\usepackage{tikz,float}

\setlist{topsep=0ex}

\makeatletter
\def\namedlabel#1#2{\begingroup
    #2%
    \def\@currentlabel{#2}%
    \phantomsection\label{#1}\endgroup
}
\makeatother

    \newcommand{\thzfc}{\mathsf{ZFC}}

    \newcommand{\Bwf}{\mathcal{B}}

    \newcommand{\Ewf}{\mathcal{E}}

    \newcommand{\Hwf}{\mathcal{H}}
    \newcommand{\Iwf}{\mathcal{I}}
    \newcommand{\Jwf}{\mathcal{J}}
    
    \newcommand{\Mwf}{\mathcal{M}}
    \newcommand{\Nwf}{\mathcal{N}}

    \newcommand{\NAwf}{\mathcal{N\!A}}
    
    \newcommand{\Pwf}{\mathcal{P}}
    \newcommand{\Scal}{\mathcal{S}}
    \newcommand{\Swf}{\mathcal{S}}

    \newcommand{\bfrak}{\mathfrak{b}}
    \newcommand{\cfrak}{\mathfrak{c}}
    \newcommand{\dfrak}{\mathfrak{d}}

    \newcommand{\Cbf}{\mathbf{C}}

    \newcommand{\Mbf}{\mathsf{M}}

    \newcommand{\Bor}{\mathbb{B}}
    
    \newcommand{\Cor}{\mathbb{C}}
    
    \newcommand{\Dor}{\mathbb{D}}
    \newcommand{\Eor}{\mathbb{E}}

    \newcommand{\Ior}{\mathbb{I}}

    \newcommand{\Por}{\mathbb{P}}
    
    \newcommand{\Qor}{\mathbb{Q}}
    
    \newcommand{\Qnm}{\dot{\mathbb{Q}}}
    \newcommand{\LOCor}{\mathds{LOC}}

    \newcommand{\menos}{\smallsetminus}
    
    \DeclareMathOperator{\pts}{\mathcal{P}}

    \newcommand{\Q}{\mathbb{Q}}
    \newcommand{\R}{\mathbb{R}}

    \newcommand{\la}{\langle}
    \newcommand{\ra}{\rangle}
    \DeclareMathOperator{\cf}{\mbox{\rm cf}}
    
    \newcommand{\frestr}{{\upharpoonright}}

    \DeclareMathOperator{\add}{\mathrm{add}}
    \DeclareMathOperator{\non}{\mbox{\rm non}}
    \DeclareMathOperator{\cov}{\mbox{\rm cov}}
    \DeclareMathOperator{\cof}{\mbox{\rm cof}}
    
    \newcommand{\leqT}{\mathrel{\mbox{$\preceq_{\mathrm{T}}$}}}
    \newcommand{\eqT}{\mathrel{\mbox{$\cong_{\mathrm{T}}$}}}
    \DeclareMathOperator{\Seq}{\mathrm{seq}}
    \DeclareMathOperator{\Seqw}{\mathrm{seq}_{<\omega}}

    \newcommand{\Ed}{\mathsf{Ed}}

    \newcommand{\seq}[2]{\la #1 :\, #2\ra}
    \newcommand{\set}[2]{\{#1 :\, #2\}}

    \newcommand{\Lc}{\mathsf{Lc}}
    \newcommand{\Cn}{\mathsf{Cn}}

    \newcommand{\blc}{\mathfrak{b}^{\mathsf{Lc}}}
    
    \newcommand{\balc}{\mathfrak{b}^{\mathsf{aLc}}}
    \newcommand{\dalc}{\mathfrak{d}^{\mathsf{aLc}}}
    \newcommand{\Fr}{\mathsf{Fr}}

    \newcommand{\Lb}{\mathsf{Lb}}
    \newcommand{\gen}{\mathrm{gn}}
    
    \newcommand{\id}{\mathrm{id}}

\newcommand{\minLc}{\mathrm{minLc}}

\newcommand{\baire}{\omega^\omega}
\newcommand{\cantor}{2^\omega}

\newcommand{\largeset}[2]{\bigg\{#1:\; #2\bigg\}}

\newcommand{\sqrb}{\sqsubset^{\bullet}}
\newcommand{\eq}{\mathsf{eq}}

\newcommand{\Crm}{\mathsf{Cv}}

    \newcommand{\lset}[2]{\left\{#1 \colon  #2\right\}}
\newcommand{\MAwf}{\mathcal{MA}}
\newcommand{\IAwf}{\mathcal{IA}}

    \definecolor{carrotorange}{rgb}{0.93, 0.57, 0.13}
    \definecolor{dodger}{rgb}{0.0,0.5,1.0}
    \newcommand{\Rrm}{\mathsf{R}}

\newcommand{\sqsubm}{\sqsubset^{\rm m}}

\definecolor{yaleblue}{rgb}{0.06, 0.3, 0.57}
\definecolor{zinnwalditebrown}{rgb}{0.17, 0.09, 0.03}
\definecolor{goldenpoppy}{rgb}{0.99, 0.76, 0.0}
\definecolor{amber}{rgb}{1.0, 0.75, 0.0}
\definecolor{amber(sae/ece)}{rgb}{1.0, 0.49, 0.0}
\definecolor{aquamarine}{rgb}{0.5, 1.0, 0.83}
\definecolor{craneorange}{RGB}{252,187,6}
\definecolor{apricot}{rgb}{0.98, 0.81, 0.69}
\definecolor{bubblegum}{rgb}{0.99, 0.76, 0.8}
\definecolor{babyblue}{rgb}{0.54, 0.81, 0.94}
\definecolor{babyblueeyes}{rgb}{0.63, 0.79, 0.95}
\definecolor{bronze}{rgb}{0.8, 0.5, 0.2}
\definecolor{brown(traditional)}{rgb}{0.59, 0.29, 0.0}
\definecolor{burgundy}{rgb}{0.5, 0.0, 0.13}
\definecolor{ao(english)}{rgb}{0.0, 0.5, 0.0}
\definecolor{asparagus}{rgb}{0.53, 0.66, 0.42}
\definecolor{celadon}{rgb}{0.67, 0.88, 0.69}
\definecolor{cadmiumgreen}{rgb}{0.0, 0.42, 0.24}
\definecolor{dartmouthgreen}{rgb}{0.05, 0.5, 0.06}
\definecolor{darkcandyapplered}{rgb}{0.64, 0.0, 0.0}
\definecolor{darkgreen}{rgb}{0.0, 0.2, 0.13}
\definecolor{ferngreen}{rgb}{0.31, 0.47, 0.26}
\definecolor{forestgreen(web)}{rgb}{0.13, 0.55, 0.13}
\definecolor{green(html/cssgreen)}{rgb}{0.0, 0.5, 0.0}
\definecolor{green(pigment)}{rgb}{0.0, 0.65, 0.31}
\definecolor{green(ncs)}{rgb}{0.0, 0.62, 0.42}
\definecolor{officegreen}{rgb}{0.0, 0.5, 0.0}
\definecolor{darkspringgreen}{rgb}{0.09, 0.45, 0.27}
\definecolor{lavenderblush}{rgb}{1.0, 0.94, 0.96}
\definecolor{goldenbrown}{rgb}{0.6, 0.4, 0.08}

\definecolor{dodger}{rgb}{0.0,0.5,1.0}
\newcommand{\dodger}[1]{{\color{dodger}#1}}

\definecolor{sub0}{RGB}{29,32,137}
\definecolor{sub1}{RGB}{1,71,157}
\definecolor{sub2}{RGB}{1,104,183}
\definecolor{sub3}{RGB}{0,160,234}
\definecolor{sug}{RGB}{0,154,68}
\definecolor{suy}{RGB}{208,219,1}

\title{Cardinal invariants associated with the combinatorics of the uniformity number of the ideal of meager-additive sets}
\author{Miguel A.~Cardona%
\thanks{Email: \href{miguel.cardona@mail.huji.ac.il}{\texttt{miguel.cardona@mail.huji.ac.il}}\newline
}
}
\date{{\normalsize
Edmond J. Safra Campus, Givat Ram\\
The Hebrew University of Jerusalem\\
Jerusalem, 91904, Israel\\
\href{mailto:miguel.cardona@mail.huji.ac.il}{\texttt{miguel.cardona@mail.huji.ac.il}}\medskip
}
}

\begin{document}

\makeatletter
\def\@roman#1{\romannumeral #1}
\makeatother

\newcounter{enuAlph}
\renewcommand{\theenuAlph}{\Alph{enuAlph}}

\numberwithin{equation}{section}
\renewcommand{\theequation}{\thesection.\arabic{equation}}

\theoremstyle{plain}
  \newtheorem{theorem}[equation]{Theorem}
  \newtheorem{corollary}[equation]{Corollary}
  \newtheorem{lemma}[equation]{Lemma}
  \newtheorem{mainlemma}[equation]{Main Lemma}
  \newtheorem{prop}[equation]{Proposition}
  \newtheorem{clm}[equation]{Claim}
  \newtheorem{fct}[equation]{Fact}
  \newtheorem{question}[equation]{Question}
  \newtheorem{problem}[equation]{Problem}
  \newtheorem{conjecture}[equation]{Conjecture}
  \newtheorem*{thm}{Theorem}
  \newtheorem{teorema}[enuAlph]{Theorem}
  \newtheorem*{corolario}{Corollary}
  \newtheorem*{scnmsc}{(SCNMSC)}
\theoremstyle{definition}
  \newtheorem{definition}[equation]{Definition}
  \newtheorem{example}[equation]{Example}
  \newtheorem{remark}[equation]{Remark}
  \newtheorem{notation}[equation]{Notation}
  \newtheorem{context}[equation]{Context}
  \newtheorem{exer}[equation]{Exercise}
  \newtheorem{exerstar}[equation]{Exercise*}

  \newtheorem*{defi}{Definition}
  \newtheorem*{acknowledgements}{Acknowledgements}
  
\def\sectionautorefname{Section}
\def\subsectionautorefname{Subsection}

\maketitle

\begin{abstract}
In~\cite{CMR2}, it was proved that it is relatively consistent that \emph{bounding number} $\bfrak$ is smaller than the uniformity of $\MAwf$, where $\MAwf$ denotes the ideal of the meager-additive sets of $\cantor$. 
To establish this result, a specific cardinal invariant, which we refer to as $\mathfrak{b}_b^\mathsf{eq}$, was introduced in close relation to Bartoszy\'nski's and Judah's characterization of the uniformity of $\MAwf$. This survey aims to explore this cardinal invariant along with its dual, which we call as $\mathfrak{d}_b^\mathsf{eq}$. In particular, we will illustrate its connections with the cardinals represented in Cicho\'n's diagram. Furthermore, we will present several open problems pertaining to these cardinals.
\end{abstract}

\section{Introduction and preliminaries}\label{sec:intro}

We first review our terminology. Given a formula $\phi$, $\forall^\infty\, n<\omega\colon \phi$ means that all but finitely many natural numbers satisfy $\phi$; $\exists^\infty\, n<\omega\colon \phi$ means that infinitely many natural numbers satisfy $\phi$. Let $\Iwf$ be an ideal of subsets of $X$ such that $\{x\}\in \Iwf$ for all $x\in X$. Throughout this paper, we demand that all ideals satisfy this latter requirement. We introduce the following four \emph{cardinal invariants associated with $\Iwf$}: 
\begin{align*}
 \add(\Iwf)&=\min\largeset{|\Jwf|}{\Jwf\subseteq\Iwf,\,\bigcup\Jwf\notin\Iwf},\\
 \cov(\Iwf)&=\min\largeset{|\Jwf|}{\Jwf\subseteq\Iwf,\,\bigcup\Jwf=X},\\
 \non(\Iwf)&=\min\set{|A|}{A\subseteq X,\,A\notin\Iwf},\textrm{\ and}\\
 \cof(\Iwf)&=\min\set{|\Jwf|}{\Jwf\subseteq\Iwf,\ \forall\, A\in\Iwf\ \exists\, B\in \Jwf\colon A\subseteq B}.
\end{align*}
These cardinals are referred to as the \emph{additivity, covering, uniformity} and \emph{cofinality of $\Iwf$}, respectively. The relationship between the cardinals defined above is illustrated in \autoref{diag:idealI}. 

\begin{figure}[ht!]
\centering
\begin{tikzpicture}[scale=1.5]
\small{
\node (azero) at (-1,1) {$\aleph_0$};
\node (addI) at (1,1) {$\add(\Iwf)$};
\node (covI) at (2,2) {$\cov(\Iwf)$};
\node (nonI) at (2,0) {$\non(\Iwf)$};
\node (cofI) at (4,2) {$\cof(\Iwf)$};
\node (sizX) at (4,0) {$|X|$};
\node (sizI) at (5,1) {$|\Iwf|$};

\draw (azero) edge[->] (addI);
\draw (addI) edge[->] (covI);
\draw (addI) edge[->] (nonI);
\draw (covI) edge[->] (sizX);
\draw (nonI) edge[->] (sizX);
\draw (covI) edge[->] (cofI);
\draw (nonI) edge[->] (cofI);
\draw (sizX) edge[->] (sizI);
\draw (cofI) edge[->] (sizI);
}
\end{tikzpicture}
\caption{Diagram of the cardinal invariants associated with $\Iwf$. An arrow  $\mathfrak x\rightarrow\mathfrak y$ means that (provably in ZFC) 
    $\mathfrak x\le\mathfrak y$.}
\label{diag:idealI}
\end{figure}
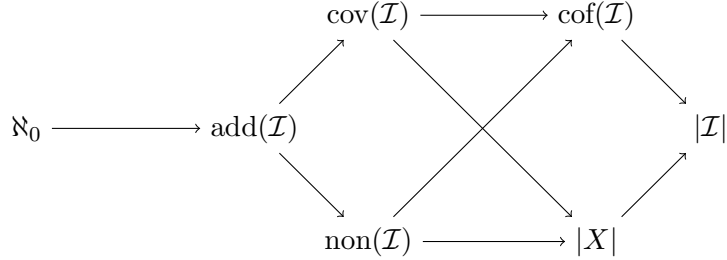

For $f,g\in\omega^\omega$ define
\[f\leq^*g\textrm{\ iff\ }\forall^\infty n\colon f(n)\leq g(n).\]
Let
\[\bfrak:=\min\{|F|:F\subseteq \omega^\omega\text{\ and }\neg\exists y\in \omega^\omega\, \forall x\in F\colon x\leq^* y\}\]
the \emph{bounding number}, and let
\[\dfrak:=\min\{|D|:D\subseteq \omega^\omega\text{\ and }\forall x\in \omega^\omega\, \exists y\in D\colon x\leq^* y\}\]
the \emph{dominating number}. As usual, $\cfrak:=2^\omega$ denotes the \emph{size of the continuum}.

\begin{definition}\label{def:a11}
Let $\Iwf\subseteq\Pwf(\cantor)$ be an ideal. 
 \begin{enumerate}[label=\rm(\arabic*)]
     \item We say that $\Iwf$ is 
\emph{translation invariant}\footnote{This paper considers the Cantor space $\cantor$ as a topological group with the standard modulo $2$ coordinatewise addition.} if $A+x\in\Iwf$ for each $A\in\Iwf$ and $x\in\cantor$.

     \item A set $X\subseteq\cantor$ is termed \emph{$\Iwf$-additive} if, for every $A\in\Iwf$, $A+X\in\Iwf$. Denote by $\IAwf$ the collection of the $\Iwf$-additive subsets of $\cantor$. Notice that $\IAwf$ is a ($\sigma$-)ideal and $\IAwf\subseteq\Iwf$ when $\Iwf$ is a translation invariant ($\sigma$-)ideal. 
 \end{enumerate}   
\end{definition}

When $\Iwf$ is either $\Mwf$ or $\Nwf$, the ideal $\IAwf$ has attracted a lot of attention. Bartoszy\'nski and Judah~\cite{bartJudah},
Pawlikowski~\cite{paw85}, Shelah~\cite{shmn},~Zindulka~\cite{zin19} and the author, Mej\'ia, and Rivera-Madrid~\cite{CMR2}, for example, were among the many who looked into them. 

Denote by $\Ior$ the set of partitions of $\omega$ into finite non-empty intervals. 

\begin{theorem}[{\cite[Thm.~13]{shmn}}]\label{thm:a0}
 Let $X\subseteq\cantor$.  Then $X\in\NAwf$ iff for all $I=\seq{I_n}{n\in\omega}\in\Ior$ there is some $\varphi\in\prod_{n\in \omega}\pts(2^{I_n})$ such that $\forall n\in \omega\colon |\varphi(n)|\leq n$ and $X\subseteq H_\varphi$,  where \[
H_\varphi:=\set{x\in\cantor}{\forall^{\infty} n\in \omega\colon x{\upharpoonright}I_n\in \varphi(n)}.
\]
\end{theorem}

The following lemma is an immediate consequence of~\autoref{def:a11}.

\begin{lemma}[{\cite[Lem.~1.3]{CMR2}}]\label{bas:NA}
For any translation invariant ideal $\Iwf$ on $\cantor$, we have:   
\begin{enumerate}[label=\rm(\arabic*)]
    \item\label{bas:NA1} $\add(\Iwf)\leq\add(\IAwf)$.

    \item\label{bas:NA2} $\non(\IAwf)\leq\non(\Iwf)$.
\end{enumerate}
\end{lemma}

The cardinal $\non(\IAwf)$ has been studied in~\cite{paw85,Kra} under the different name \emph{transitive additivity of $\Iwf$}:\footnote{In~\cite{BJ} is denoted by $\add^\star(\Iwf)$.}
\[\add_t^*(\Iwf)=\min\set{|X|}{X\subseteq\cantor\text{ and }\exists A\in\Iwf\colon A+X\notin\Iwf}.\]
It is clear from the definition that $\non(\IAwf)=\add_t^*(\Iwf)$.

Pawlikowski~\cite{paw85} characterized $\add^*_t(\Nwf)$ (i.e.\ $\non(\NAwf)$) employing slaloms.

\begin{definition}\label{def:a0}
Given a sequence of non-empty sets $b = \seq{b(n)}{n\in\omega}$ and $h\colon \omega\to\omega$, define
\begin{align*}
 \prod b &:= \prod_{n\in\omega}b(n),\textrm{\ and} \\
 \Swf(b,h) &:= \prod_{n\in\omega} [b(n)]^{\leq h(n)}.
\end{align*}
For two functions $x\in\prod b$ and $\varphi\in\Swf(b,h)$ write  
\[x\,\in^*\varphi\textrm{\ iff\ }\forall^\infty n\in\omega:x(n)\in \varphi(n).\]
\end{definition}

\begin{theorem}[{\cite[Lem.~2.2]{paw85}, see also~\cite[Thm.~8.3]{CMlocalc}}]
$\non(\NAwf)$ is the size of the minimal bounded family $F\subseteq\baire$  such that 
\[\forall\varphi\in\Swf(b,\id_\omega)\,\exists x\in F\colon x\,\not\in^*\varphi.\]
\end{theorem}

Stated differently, the uniformity of $\NAwf$ can be described using localization cardinals as follows. 

For $b$ and $h$ as in~\autoref{def:a0}, define 
\[\blc_{b,h}:=\min\largeset{|F|}{F\subseteq \prod b \text{ and } \neg\exists \varphi \in \Swf(b,h)\,\forall x \in F\colon x\in^* \varphi},\]
and set $\minLc:=\min\set{\blc_{b,\id_\omega}}{b\in\baire}$. Here, $\id_\omega$ denotes the identity function on $\omega$.    

Hence, we obtain that $\non(\NAwf)=\minLc$. Another characterization of $\minLc$ is the following.

\begin{lemma}[{\cite[Lemma~3.8]{CM}}]\label{minlcCM}
$\minLc=\min\set{\blc_{b,h}}{b\in\baire}$ when $h$ goes to infinity.  
\end{lemma}

Hence, we can infer:

\begin{corollary}
 $\non(\NAwf)=\min\set{\blc_{b,h}}{b\in\baire}$ when $h$ goes to infinity.
\end{corollary}

Moreover, it recently was proved that 

\begin{lemma}[{\cite[Thm.~A]{CMR2}}]
$\non(\NAwf)=\add(\NAwf)$.    
\end{lemma}

The characterization of $\add(\Nwf)$ by Pawlikowski can be expressed as follows as a direct result of the previous result:

\begin{theorem}[{\cite[Lem.~2.3]{paw85}}]\label{chPawmn}\
$\add(\Nwf)=\min\{\bfrak,\add(\NAwf)\}$. 

\end{theorem}

We below focus on the $\sigma$-ideal of meager-additive sets and its uniformity. We start with the characterization for $\MAwf$ because of Bartoszy\'nski and Judah, just like in~\autoref{thm:a0}. 

\begin{theorem}[{\cite[Thm.~2.2]{bartJudah}}]\label{thm:a1}
Let $X\subseteq\cantor$.  Then $X\in\MAwf$ iff for all $I\in\Ior$ there are $J\in\Ior$ and $y\in\cantor$ such that \[\forall x\in X\, \forall^\infty n<\omega\, \exists k<\omega\colon I_k\subseteq J_n\text{\ and\ }x{\upharpoonright}I_k=y{\upharpoonright}I_k.\] Furthermore, Shelah~\cite[Thm.~18]{shmn} proved that $J$ can be found coarser than $I$, i.e, all members of $J$ are the union of members of $I$   
\end{theorem}

They also established a characterization of the uniformity of the meager-additive ideal:

 \begin{theorem}[{\cite[Thm.~2.2]{bartJudah}}, see also {\cite[Thm.~2.7.14]{BJ}}]\label{addtch}\ \\ 
The cardinal $\non(\MAwf)$ is the largest cardinal $\kappa$ such that, for every bounded family $F\subseteq\baire$ of size ${<}\kappa$, \[\tag{\faPagelines}\exists r,h\in\baire\, \forall f\in F\, \exists n\in\omega\, \forall m\geq n\,  \exists k\in[r(m),r(m+1)]\colon f(k)=h(k).\label{addtchtag}\]
\end{theorem}

We below introduce two cardinal invariants motivated by~\eqref{addtchtag}, which were introduced by by the author along with Mej\'ia and Rivera-Madrid in~\cite{CMR2}.

\begin{definition}
Let $b\in\baire$. For  $I\in\Ior$, and for $f, h\in\prod b$, define 

    \[f  \sqrb (I,h)\textrm{\ iff\ }\forall^\infty n\in\omega\,\exists k\in I_n\colon f(k)=h(k).\]
   
We define the following cardinal invariants associated with $\sqrb$.
\[\bfrak^{\eq}_b:=\min\set{|F|}{F\subseteq \prod b\textrm{\ and\ }\neg\exists I\in\Ior\,\exists h\in\prod b\,\forall f\in F\colon f  \sqrb (I,h)}\]
and 
\[\dfrak^{\eq}_b:=\min\set{|D|}{D\subseteq\Ior\times\prod b\textrm{\ and\ }\forall f\in \prod b\,\exists (I,h)\in D\colon f  \sqrb (I,h)}.\]    
\end{definition}

The study of uniformity of $\MAwf$ was better understood due to these cardinals, which for instance, were utilized by the author along with Mej\'ia and Rivera-Madrid~\cite{CMR2} to prove the consistency of $\non(\MAwf)>\bfrak$ and $\cov(\MAwf)<\non(\Nwf)$.

It also turns out that~\autoref{addtch} can be reformulated as
\[\tag{\faLeaf}\non(\MAwf)=\min\set{\bfrak^{\eq}_b}{b\in\baire}.\label{la:addtch}\] 
To be thorough, we provide a proof of~\eqref{la:addtch} (see~\autoref{RsAMone} and~\autoref{prodRb}).

This survey aims to study the cardinals invariants $\bfrak^{\eq}_b$ and $\dfrak^{\eq}_b$, so  one of the goal of this article is to establish:

\begin{teorema}\label{thm:a4}
The following relations in~\autoref{cichonplus} hold, where $\mathfrak x\rightarrow\mathfrak y$ means 
    $\mathfrak x\le\mathfrak y$.
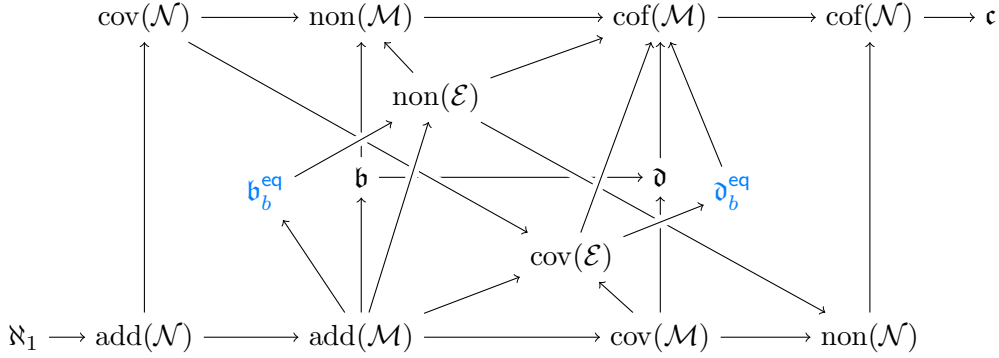
\begin{figure}[H]
\centering
\begin{tikzpicture}[scale=1.06]
\small{
\node (aleph1) at (-1,3) {$\aleph_1$};
\node (addn) at (0.5,3){$\add(\Nwf)$};
\node (covn) at (0.5,7){$\cov(\Nwf)$};
\node (nonn) at (9.5,3) {$\non(\Nwf)$} ;
\node (cfn) at (9.5,7) {$\cof(\Nwf)$} ;
\node (addm) at (3.19,3) {$\add(\Mwf)$} ;
\node (covm) at (6.9,3) {$\cov(\Mwf)$} ;
\node (nonm) at (3.19,7) {$\non(\Mwf)$} ;
\node (cfm) at (6.9,7) {$\cof(\Mwf)$} ;
\node (b) at (3.19,5) {$\bfrak$};
\node (d) at (6.9,5) {$\dfrak$};
\node (c) at (11,7) {$\cfrak$};
\node (e) at (2,4.8) {\dodger{$\bfrak_b^\eq$}};
\node (deq) at (7.8,4.8) {\dodger{$\dfrak_b^\eq$}};
\node (none) at (4.12,6) {$\non(\Ewf)$};
\node (cove) at (5.8,4) {$\cov(\Ewf)$};
\draw (aleph1) edge[->] (addn)
      (addn) edge[->] (covn)
      (covn) edge [->] (nonm)
      (nonm)edge [->] (cfm)
      (cfm)edge [->] (cfn)
      (cfn) edge[->] (c);

\draw
   (addn) edge [->]  (addm)
   (addm) edge [->]  (covm)
   (covm) edge [->]  (nonn)
   (nonn) edge [->]  (cfn);
\draw (addm) edge [->] (b)
      (b)  edge [->] (nonm);
\draw (covm) edge [->] (d)
      (d)  edge[->] (cfm);
\draw (b) edge [->] (d);

\draw   (none) edge [->] (nonm)
        (none) edge [->] (cfm)
        (addm) edge [->] (cove);
      
\draw (none) edge [line width=.15cm,white,-] (nonn)
      (none) edge [->] (nonn);
      
\draw (cove) edge [line width=.15cm,white,-] (covn)
      (cove) edge [<-] (covm)
      (cove) edge [<-] (covn);

\draw (addm) edge [line width=.15cm,white,-] (none)
      (addm) edge [->] (none); 

\draw (cove) edge [line width=.15cm,white,-] (cfm)
      (cove) edge [->] (cfm);

\draw (e) edge [line width=.15cm,white,-] (none)
      (e) edge [->] (none);   

\draw (e) edge [line width=.15cm,white,-] (addm)
      (e) edge [<-] (addm);  


\draw (deq) edge [line width=.15cm,white,-] (cove)
      (deq) edge [<-] (cove);  

\draw (deq) edge [line width=.15cm,white,-] (cfm)
      (deq) edge [->] (cfm);  

}
\end{tikzpicture}
\caption{Including $\bfrak^{\eq}_b$ and  $\dfrak^{\eq}_b$ to Cicho\'n's diagram.}
\label{cichonplus}
\end{figure}
\end{teorema}

\begin{teorema}\label{thm:a5}
The following relations in~\autoref{cichonplus+} hold, where $\mathfrak x\rightarrow\mathfrak y$ means 
    $\mathfrak x\le\mathfrak y$.
\begin{figure}[H]
\centering
\begin{tikzpicture}[scale=1.06]
\small{
\node (aleph1) at (-1,3) {$\aleph_1$};
\node (addn) at (0.5,3){$\add(\Nwf)$};
\node (covn) at (0.5,7){$\cov(\Nwf)$};
\node (nonn) at (9.5,3) {$\non(\Nwf)$} ;
\node (cfn) at (9.5,7) {$\cof(\Nwf)$} ;
\node (addm) at (3.19,3) {$\add(\Mwf)$} ;
\node (covm) at (6.9,3) {$\cov(\Mwf)$} ;
\node (nonm) at (3.19,7) {$\non(\Mwf)$} ;
\node (cfm) at (6.9,7) {$\cof(\Mwf)$} ;
\node (b) at (3.19,5) {$\bfrak$};
\node (d) at (6.9,5) {$\dfrak$};
\node (c) at (11,7) {$\cfrak$};
\node (e) at (2,4.8) {\dodger{$\bfrak_b^\eq$}};
\node (deq) at (7.8,4.8) {\dodger{$\dfrak_b^\eq$}};
\node (none) at (4.12,6) {$\non(\Ewf)$};
\node (cove) at (5.8,4) {$\cov(\Ewf)$};
\draw (aleph1) edge[->] (addn)
      (addn) edge[->] (covn)
      (covn) edge [->] (nonm)
      (nonm)edge [->] (cfm)
      (cfm)edge [->] (cfn)
      (cfn) edge[->] (c);

\draw
   (addn) edge [->]  (addm)
   (addm) edge [->]  (covm)
   (covm) edge [->]  (nonn)
   (nonn) edge [->]  (cfn);
\draw (addm) edge [->] (b)
      (b)  edge [->] (nonm);
\draw (covm) edge [->] (d)
      (d)  edge[->] (cfm);
\draw (b) edge [->] (d);

\draw   (none) edge [->] (nonm)
        (none) edge [->] (cfm)
        (addm) edge [->] (cove);
      
\draw (none) edge [line width=.15cm,white,-] (nonn)
      (none) edge [->] (nonn);
      
\draw (cove) edge [line width=.15cm,white,-] (covn)
      (cove) edge [<-] (covm)
      (cove) edge [<-] (covn);

\draw (addm) edge [line width=.15cm,white,-] (none)
      (addm) edge [->] (none); 

\draw (cove) edge [line width=.15cm,white,-] (cfm)
      (cove) edge [->] (cfm);

\draw (e) edge [line width=.15cm,white,-] (covn)
      (e) edge [<-] (covn);

\draw (e) edge [line width=.15cm,white,-] (nonm)
      (e) edge [->] (nonm);

\draw (e) edge [line width=.15cm,white,-] (addm)
      (e) edge [<-] (addm);  


\draw (deq) edge [line width=.15cm,white,-] (covm)
      (deq) edge [<-] (covm);

      \draw (deq) edge [line width=.15cm,white,-] (nonn)
      (deq) edge [->] (nonn);

\draw (deq) edge [line width=.15cm,white,-] (cfm)
      (deq) edge [->] (cfm);  

}
\end{tikzpicture}
\caption{Including $\bfrak^{\eq}_b$ and  $\dfrak^{\eq}_b$ to Cicho\'n's diagram. Additionally, if $\sum_{k<\omega}\frac{1}{b(k)} = \infty$ then $\cov(\Nwf)\leq\bfrak_b^\eq$ and $\dfrak_b^\eq\leq\non(\Nwf)$. }
\label{cichonplus+}
\end{figure}
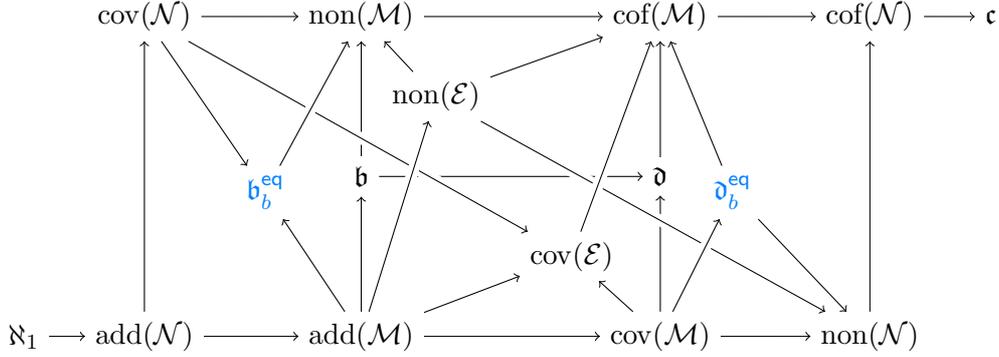
\end{teorema}

\autoref{thm:a4}-\ref{thm:a5} will be proved in~\autoref{sec:zfc}. In~\autoref{sec:cons}, we present one forcing notion closely
related to the $\bfrak_b^\eq$, which we call $\Por_b$ and illustrate the effect of iterating $\Por_b$ on Cicho\'n’s diagram. In addition, we prove the following:

\begin{teorema}[\autoref{thm:c2}]\label{thm:a6}
Let $\aleph_1\leq \lambda_1\leq \lambda_2\leq \lambda_3 \leq \lambda_4$ be regular cardinals, and assume $\lambda_5$ is a cardinal such that $\lambda_5\geq\lambda_4$ $\lambda_5= \lambda_5^{\aleph_0}$ and $\cf([\lambda_5]^{<\lambda_i}) = \lambda_5$ for $i=1,\ldots,3$. Then, we can construct a FS iteration of length $\lambda_5$ of ccc posets forcing~\autoref{Figthm:a6}.
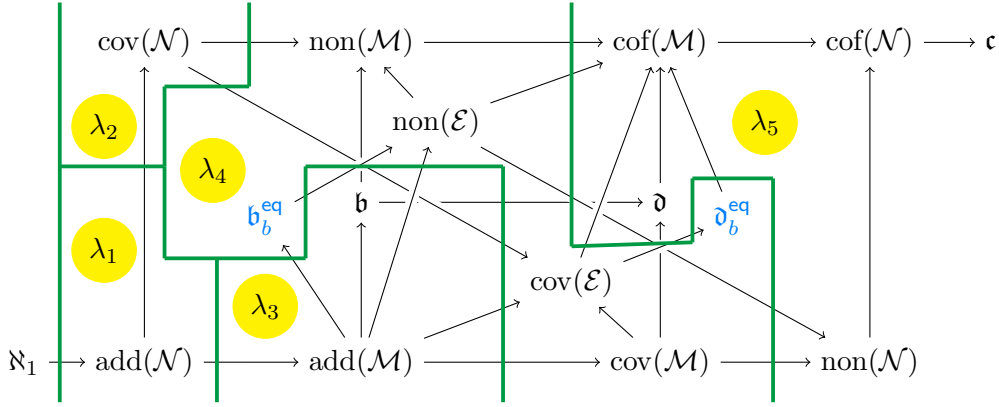
\begin{figure}[H]
\centering
\begin{tikzpicture}[scale=1.06]
\small{
\node (aleph1) at (-1,3) {$\aleph_1$};
\node (addn) at (0.5,3){$\add(\Nwf)$};
\node (covn) at (0.5,7){$\cov(\Nwf)$};
\node (nonn) at (9.5,3) {$\non(\Nwf)$} ;
\node (cfn) at (9.5,7) {$\cof(\Nwf)$} ;
\node (addm) at (3.19,3) {$\add(\Mwf)$} ;
\node (covm) at (6.9,3) {$\cov(\Mwf)$} ;
\node (nonm) at (3.19,7) {$\non(\Mwf)$} ;
\node (cfm) at (6.9,7) {$\cof(\Mwf)$} ;
\node (b) at (3.19,5) {$\bfrak$};
\node (d) at (6.9,5) {$\dfrak$};
\node (c) at (11,7) {$\cfrak$};
\node (e) at (2,4.8) {\dodger{$\bfrak_b^\eq$}};
\node (deq) at (7.8,4.8) {\dodger{$\dfrak_b^\eq$}};
\node (none) at (4.12,6) {$\non(\Ewf)$};
\node (cove) at (5.8,4) {$\cov(\Ewf)$};
\draw (aleph1) edge[->] (addn)
      (addn) edge[->] (covn)
      (covn) edge [->] (nonm)
      (nonm)edge [->] (cfm)
      (cfm)edge [->] (cfn)
      (cfn) edge[->] (c);

\draw
   (addn) edge [->]  (addm)
   (addm) edge [->]  (covm)
   (covm) edge [->]  (nonn)
   (nonn) edge [->]  (cfn);
\draw (addm) edge [->] (b)
      (b)  edge [->] (nonm);
\draw (covm) edge [->] (d)
      (d)  edge[->] (cfm);
\draw (b) edge [->] (d);

\draw   (none) edge [->] (nonm)
        (none) edge [->] (cfm)
        (addm) edge [->] (cove);
      
\draw (none) edge [line width=.15cm,white,-] (nonn)
      (none) edge [->] (nonn);
      
\draw (cove) edge [line width=.15cm,white,-] (covn)
      (cove) edge [<-] (covm)
      (cove) edge [<-] (covn);

\draw (addm) edge [line width=.15cm,white,-] (none)
      (addm) edge [->] (none); 

\draw (cove) edge [line width=.15cm,white,-] (cfm)
      (cove) edge [->] (cfm);

\draw (e) edge [line width=.15cm,white,-] (none)
      (e) edge [->] (none);   

\draw (e) edge [line width=.15cm,white,-] (addm)
      (e) edge [<-] (addm);  


\draw (deq) edge [line width=.15cm,white,-] (cove)
      (deq) edge [<-] (cove);  

\draw (deq) edge [line width=.15cm,white,-] (cfm)
      (deq) edge [->] (cfm);  

\draw[color=sug,line width=.05cm] (0.75,4.3)--(2.5,4.3); 
\draw[color=sug,line width=.05cm] (2.5,5.45)--(4.95,5.45);  
\draw[color=sug,line width=.05cm] (-0.55,5.45)--(0.75,5.45); 
\draw[color=sug,line width=.05cm] (0.75,6.45)--(1.8,6.45);   
\draw[color=sug,line width=.05cm] (5.8,4.45)--(5.8,7.5);
\draw[color=sug,line width=.05cm] (7.3,4.5)--(7.3,5.3);
\draw[color=sug,line width=.05cm] (8.3,2.5)--(8.3,5.3);

\draw[color=sug,line width=.05cm] (4.95,2.5)--(4.95,5.45);
\draw[color=sug,line width=.05cm] (-0.55,2.5)--(-0.55,7.5);
\draw[color=sug,line width=.05cm] (1.8,6.45)--(1.8,7.5);
\draw[color=sug,line width=.05cm] (0.75,5.45)--(0.75,6.45); 
\draw[color=sug,line width=.05cm] (0.75,4.3)--(0.75,5.45);
\draw[color=sug,line width=.05cm] (1.4,2.5)--(1.4,4.3);
\draw[color=sug,line width=.05cm] (2.5,4.3)--(2.5,5.45);

\draw[color=sug,line width=.05cm] (5.8,4.45)--(7.3,4.5);
\draw[color=sug,line width=.05cm] (7.3,5.3)--(8.3,5.3);

\draw[circle, fill=yellow,color=yellow] (1.35,5.4) circle (0.4);
\draw[circle, fill=yellow,color=yellow] (0,4.4) circle (0.4);
\draw[circle, fill=yellow,color=yellow] (0,5.95) circle (0.4);
\draw[circle, fill=yellow,color=yellow] (2,3.7) circle (0.4);
\draw[circle, fill=yellow,color=yellow] (8.2,6) circle (0.4);
\node at (8.2,6) {$\lambda_5$};
\node at (1.35,5.4) {$\lambda_4$};
\node at (0,4.4) {$\lambda_1$};
\node at (0,5.95) {$\lambda_2$};
\node at (2,3.7) {$\lambda_3$};
}
\end{tikzpicture}
\caption{The constellation of Cicho\'n's diagram forced in \autoref{thm:a6}.}
\label{Figthm:a6}
\end{figure}
\end{teorema}

We close this section by reviewing everything needed to develop our targets.

\begin{definition}\label{def:relsys}
We say that $\Rrm=\la X, Y, \sqsubset\ra$ is a \textit{relational system} if it consists of two non-empty sets $X$ and $Y$ and a relation $\sqsubset$.
\begin{enumerate}[label=(\arabic*)]
    \item A set $F\subseteq X$ is \emph{$\Rrm$-bounded} if $\exists\, y\in Y\ \forall\, x\in F\colon x \sqsubset y$. 
    \item A set $E\subseteq Y$ is \emph{$\Rrm$-dominating} if $\forall\, x\in X\ \exists\, y\in E\colon x \sqsubset y$. 
\end{enumerate}

We associate two cardinal invariants with this relational system $\Rrm$: 
\begin{itemize}
    \item[{}] $\bfrak(\Rrm):=\min\{|F|\colon  F\subseteq X  \text{ is }\Rrm\text{-unbounded}\}$ the \emph{unbounding number of $\Rrm$}, and
    
    \item[{}] $\dfrak(\Rrm):=\min\{|D|\colon  D\subseteq Y \text{ is } \Rrm\text{-dominating}\}$ the \emph{dominating number of $\Rrm$}.
\end{itemize}
\end{definition}

Note that $\dfrak(\Rrm)=1$ iff $\bfrak(\Rrm)$ is undefined (i.e.\ there are no $\Rrm$-unbounded sets, which is the same as saying that $X$ is $\Rrm$-bounded).  Dually, $\bfrak(\Rrm)=1$ iff $\dfrak(\Rrm)$ is undefined (i.e.\ there are no $\Rrm$-dominating families).

Directed preorders provide a very representative broad example of relational systems.

\begin{definition}\label{examSdir}
We say that $\la S,\leq_S\ra$ is a \emph{directed preorder} if it is a preorder (i.e.\ $\leq_S$ is a reflexive and transitive relation on $S$) such that 
\[\forall\, x, y\in S\ \exists\, z\in S\colon x\leq_S z\text{ and }y\leq_S z.\] 
A directed preorder $\la S,\leq_S\ra$ is seen as the relational system $S=\la S, S,\leq_S\ra$, and its associated cardinal invariants are denoted by $\bfrak(S)$ and $\dfrak(S)$. The cardinal $\dfrak(S)$ is actually the \emph{cofinality of $S$}, typically denoted by $\cof(S)$ or $\cf(S)$.
\end{definition}



\begin{example}\label{Charb-d}
Define the following relation on $\Ior$:
   \[ I \sqsubseteq J \text{ iff } \forall^\infty n<\omega\, \exists m<\omega\colon I_m\subseteq J_n.
   \]
Note that $\sqsubseteq$ is a directed preorder on $\Ior$, so we think of $\Ior$ as the relational system with the relation $\sqsubseteq$. 
In Blass~\cite{blass}, it is proved that $\Ior \eqT \omega^\omega$. Hence, $\bfrak=\bfrak(\Ior)$ and $\dfrak=\dfrak(\Ior)$. 
\end{example}

\begin{example}\label{exm:Iwf}
We consider the following relational systems for any ideal $\Iwf$ on $X$. 
\begin{enumerate}[label=(\arabic*)]
    \item $\Iwf:=\la\Iwf,\subseteq\ra$ is a directed partial order. Note that 
    \begin{align*}
        \bfrak(\Iwf) & =\add(\Iwf) \\ 
        \dfrak(\Iwf) & =\cof(\Iwf)
    \end{align*}
    
    \item $\Cbf_\Iwf:=\la X,\Iwf,\in\ra$. 
    When $\bigcup\Iwf = X$,
    \begin{align*}
        \bfrak(\Cbf_\Iwf) & =\non(\Iwf) \\
        \dfrak(\Cbf_\Iwf) & =\cov(\Iwf) 
    \end{align*}    
\end{enumerate}
\end{example}

\begin{example}\label{exa:a2}
For $b\in\omega^\omega$ define the relational system $\Rrm_b:=\la\prod b,\Ior\times\prod b,\sqrb\ra$.  Notice that
 \begin{enumerate}[label=\rm(\arabic*)]
    \item  $\Rrm_b\eqT\la\prod b,\Ior\times\baire,\sqrb\ra$. Then $\bfrak(\Rrm_b)=\bfrak_{b}^\eq$ and $\dfrak(\Rrm_b)=\dfrak_{b}^\eq$.   
     \item If $b'\in\omega^\omega$ and
$b\leq^* b'$, then $\Rrm_b\leqT\Rrm_{b'}$. In particular, $\bfrak_{b'}^\eq \leq \bfrak_b^\eq$ and $\dfrak_b^\eq\leq \dfrak_{b'}^\eq$. 
\end{enumerate} 
\end{example}

\begin{remark}
If $b\ngeq^*2$ then we can find some $(I,h)\in \Ior\times \prod b$ such that $f\sqrb (I,h)$ for all $f\in\prod b$, so $\dfrak(\Rrm_b) = 1$ and $\bfrak(\Rrm_b)$ is undefined.    
\end{remark}

We now review the products of relational systems.

\begin{definition}\label{def:prodrel}
Let $\overline{\Rrm} = \seq{\Rrm_i}{i\in K}$ be a sequence of relational systems $\Rrm_i = \la X_i,Y_i,\sqsubset_i\ra$. Define $\prod \overline{\Rrm} = \prod_{i\in K} \Rrm_i:=\left\la \prod_{i\in K}X_i, \prod_{i\in K}Y_i, \sqsubset^\times \right\ra$ where $x \sqsubset^\times y$ iff $x_i\sqsubset_i y_i$ for all $i\in K$.

For two relational systems $\Rrm$ and $\Rrm'$, write $\Rrm\times \Rrm'$ to denote their product, and when $\Rrm_i = \Rrm$ for all $i\in K$, we write $\Rrm^K := \prod \overline{\Rrm}$.
\end{definition}

\begin{fct}[{\cite{CarMej23}}]\label{products}
Let $\overline{\Rrm}$ be as in \autoref{def:prodrel}. Then $\sup_{i\in K}\dfrak(\Rrm_i)\leq\dfrak(\prod \overline{\Rrm})\leq\prod_{i\in K}\dfrak(\Rrm_i)$ and $\bfrak(\prod \overline{\Rrm})=\min_{i\in K}\bfrak(\Rrm_i)$.
\end{fct}

We use the composition of relational systems to prove~\autoref{dRbup}.

\begin{definition}[{\cite[Sec.~4]{blass}}]\label{compTK}
Let $\Rrm_e=\la X_e,Y_e,\sqsubset_e\ra$ be a relational system for $e\in\{0,1\}$. 
The \emph{composition of $\Rrm_0$ with $\Rrm_1$} is defined by $(\Rrm_0;\Rrm_1):=\la X_0\times X_1^{Y_0}, Y_0\times Y_1, \sqsubset_* \ra$ where
\[(x,f)\sqsubset_*(y,b)\text{ iff }x \sqsubset_0 y\text{ and }f(y) \sqsubset_1 b.\]
\end{definition}

\begin{fct}\label{ex:compleqT}
Let $\Rrm_i$ be a relational system for $i<3$. If $\Rrm_0\leqT\Rrm_1$, then $\Rrm_0\leqT \Rrm_1\times \Rrm_2 \leqT (\Rrm_1,\Rrm_2)$ and 
$\Rrm_1\times \Rrm_2
\eqT \Rrm_2\times \Rrm_1$.
\end{fct}

The following theorem describes the effect of the composition on cardinal invariants.

\begin{theorem}[{\cite[Thm.~4.10]{blass}}]\label{blascomp}
Let $\Rrm_e$ be a relational system for $e\in\{0,1\}$. Then $\bfrak(\Rrm_0;\Rrm_1) =\min\{\bfrak(\Rrm_0),\bfrak(\Rrm_1)\}$ and $\dfrak(\Rrm_0;\Rrm_1)=\dfrak(\Rrm_0)\cdot\dfrak(\Rrm_1)$. 
\end{theorem}

Instead of dealing with all meager sets, we will consider a suitably chosen cofinal family below.

\begin{definition}
Let $I\in\Ior$ and let $x\in\cantor$. Define 
\[B_{x,I}:=\set{y\in\cantor}{\forall^\infty n\in\omega\colon y{\upharpoonright}I_n\neq x{\upharpoonright}I_n}.\]
For $n\in\omega$, define
\[B_{x,I}^n:=\set{y\in\cantor}{\forall m\geq n\colon x{\upharpoonright}I_m\neq y{\upharpoonright}I_m}.\]
Then $B_{x,I}^m\subseteq B_{x,I}^n$ whenever $m < n <\omega$. Thus, $B_{x,I}=\bigcup_{n\in\omega}B_{x,I}^n$. 

Denote by $B_{I}$ the set $B_{0,I}=\set{y\in\cantor}{\forall^\infty n\in\omega\colon y{\upharpoonright}I_n\neq0)}$.    
\end{definition}

A pair $(x,I)\in 2^\omega\times \Ior$ is known as a \emph{chopped real}, and these are used to produce a cofinal family of meager sets.
It is clear that $B_{x,I}$ is a meager subset 
of $\cantor$ (see, e.g.~\cite{blass}).

\begin{theorem}[{Talagrand~\cite{Tal98}}, see also e.g.~{\cite[Prop.~13]{BWS}}]\label{ChM}
For each meager set $F\subseteq\cantor$ and $I\in\Ior$ there are $x\in\cantor$ and $I'\in\Ior$ such that $F\subseteq B_{I',x}$ and each $I'_n$ is the union of finitely many $I_k$'s.
\end{theorem}

\begin{lemma}[{\cite[Prop~9]{BWS}}]\label{baseM}
For $x, y\in\cantor$ and for $I, J\in\Ior$, the following statements are equivalent:
\begin{enumerate}[label=\normalfont(\arabic*)]
    \item $B_{I,x}\subseteq B_{J,y}$.
    \item $\forall^\infty n<\omega\,\exists k<\omega\colon I_k\subseteq J_n$ and $x{\upharpoonright}I_k=y{\upharpoonright}I_k$.
\end{enumerate}
\end{lemma}

\begin{definition}
Given a sequence $b=\seq{b(n)}{n\in \omega}$ of non-empty sets, denote 
\[\Seq_{<\omega} b:=\bigcup_{n<\omega}\prod_{i<n}b(i).\] 
For each $\sigma\in\Seq_{<\omega}(b)$ define
\[[s]:=[s]_b:=\set{x\in\prod b}{ s\subseteq x}.\]    
\end{definition}

As a topological space, $\prod b$ has the product topology with each $b(n)$ endowed with the discrete topology. Note that $\set{[s]_b}{s\in \Seqw b}$ forms a base of clopen sets for this topology. When each $b(n)$ is countable we have that $\prod b$ is a Polish space and, in addition, if $|b(n)|\geq 2$ for infinitely many $n$, then $\prod b$ is a perfect Polish space. The most relevant instances are:
\begin{itemize}
   \item The Cantor space $2^\omega$, when $b(n)=2$ for all $n$, and 
   \item The Baire space $\omega^\omega$, when $b(n)=\omega$ for all $n$.
\end{itemize}
We now review the Lebesgue measure on $\prod b$ when each $b(n)\leq\omega$ is an ordinal. For any ordinal $0<\eta\leq\omega$, the probability measure $\mu_\eta$ on the power set of $\eta$ is defined by:
\begin{itemize}
\item when $\eta=n<\omega$, $\mu_{n}$ is the measure such that, for all $k<n$,
$\mu_{n}(\{k\})=\frac{1}{n}$, and 
\item when $\eta=\omega$, $\mu_{\omega}$ is the measure such that, for  $k<\omega$,
$\mu_{\omega}(\{k\})=\frac{1}{2^{k+1}}$.
\end{itemize}
Denote by $\Lb_b$ the product measure of $\la \mu_{b(n)}:\, n<\omega\ra$, which we call \emph{the Lebesgue Measure on $\prod b$}, so $\Lb_b$ is a  probability measure on the Borel $\sigma$-algebra of $\prod b$. More concretely, $\Lb_{b}$ 
is the unique measure on the Borel $\sigma$-algebra such that, for any $s\in \Seqw b$, $\Lb_b([s])=\prod_{i<|s|}\mu_{b(i)}(\{s(i)\})$. In particular, denote by $\Lb$, $\Lb_{2}$ and $\Lb_{\omega}$ the Lebesgue measure on $\R$, on $2^\omega$, and on $\omega^\omega$, respectively. 

Let $X$ be a topological space. Denote by $\Mwf(X)$ the collection of all meager subsets of $X$, and let $\Mwf:=\Mwf(\R)$.  If $X$ is a perfect Polish space, then $\Mwf(X)\eqT \Mwf(\R)$ and $\Crm_{\Mwf(X)}\eqT \Crm_{\Mwf(\R)}$ (see~\cite[Ex.~8.32 \&~Thm.~15.10]{Ke2}). Therefore, the cardinal invariants associated with the meager ideal are independent of the perfect Polish space used to calculate it. When the space is clear from the context, we write $\Mwf$ for the meager ideal.

On the other hand, denote by $\Bwf(X)$ the $\sigma$-algebra of Borel subsets of $X$, and assume that $\mu\colon\Bwf(X)\to [0,\infty]$ is a $\sigma$-finite measure such that $\mu(X)>0$ and every singleton has measure zero.
    Denote by $\Nwf(\mu)$ the ideal generated by the $\mu$-measure zero sets, which is also denoted by $\Nwf(X)$ when the measure on $X$ is clear.
    Then $\Nwf(\mu)\eqT \Nwf(\Lb)$ and $\Crm_{\Nwf(\mu)}\eqT \Crm_{\Nwf(\Lb)}$ where $\Lb$ is the Lebesgue measure on $\R$ (see~\cite[Thm.~17.41]{Ke2}). Therefore, the four cardinal invariants associated with both measure zero ideals are the same. When $b=\la b(n):\, n<\omega\ra$, each $b(n)\leq \omega$ is a non-zero ordinal, and $\prod b$ is perfect, we have that $\Lb_b$ satisfies the properties of $\mu$ above. When the measure space is understood, we just write $\Nwf$ for the null ideal.

\begin{definition}
    For $b$ as above, denote by $\Ewf(\prod b)$ the ideal generated by the $F_\sigma$ measure zero subsets of $\prod b$. Likewise, define $\Ewf(\R)$ and $\Ewf([0,1])$.
When $\prod b$ is perfect, the map $F_b\colon \prod b\to[0,1]$ defined by
\[F_b(x):=\sum_{n<\omega}\frac{x(n)}{\prod_{i\leq n}b(i)}\]
is a continuous onto function, and it preserves measure. Hence, this map preserves sets between $\Ewf(\prod b)$ and $\Ewf([0,1])$ via images and pre-images. Therefore, $\Ewf(\prod b)\eqT \Ewf([0,1])$ and $\Crm_{\Ewf(\prod b)} \eqT \Crm_{\Ewf([0,1])}$. We also have $\Ewf(\R)\eqT \Ewf([0,1])$ and $\Crm_{\R} \eqT \Crm_{\Ewf([0,1])}$, as well as $\Ewf(\omega^\omega)\eqT \Ewf(2^\omega)$ and $\Crm_{\Ewf(\omega^\omega)}\eqT \Crm_{\Ewf(2^\omega)}$.

When the space is clear, we write $\Ewf$.
Therefore, the cardinal invariants of $\Ewf$ do not depend on the previous spaces.
\end{definition}

\section{ZFC results}\label{sec:zfc}

This section aims to display the new arrows that appear in Cichon’s diagram. All of the contents in this section are taken almost verbatim from~\cite[Sec.~2]{CMR2}.

 \begin{lemma}\label{lem:b0}
 $\Crm_\Mwf\leqT\Rrm_b$ whenever $b\geq^* 2$. In particular, $\bfrak_b^\eq\leq\non(\Mwf)$ and $\cov(\Mwf)\leq \dfrak_b^\eq$. 
\end{lemma}
\begin{proof}
We work with $\Mwf(\prod b)$ instead of $\Mwf$ (see~\autoref{sec:intro}). Let $F\colon \prod b\to \prod b$ be the identity function and define $G\colon \Ior\times\prod b \rightarrow \Mwf\big(\prod b\big)$ as follows.
\begin{align*}
G\colon \Ior\times\prod b & \rightarrow \Mwf\big(\prod b\big) \\
        (J,h) & \mapsto  \set{x\in \prod b}{x \sqrb (J,h)}
\end{align*} 
Observe that $\set{x\in \prod b}{x \sqrb (J,h)}\in\Mwf(\prod b)$, since    
    \[\set{x\in \prod b}{x \sqrb (J,h)} = \bigcup_{m<\omega}\bigcap_{n\geq m}\bigcup_{k\in J_n}A^{h(k)}_k,\]
\text{ where\ }$A^{\ell}_k := \set{x\in \prod b}{x(k) = \ell}$ for $\ell<b(k)$,  and each $A^\ell_k$ is clopen. In fact, it is $F_\sigma$-set. It is clear that if $x \sqrb (J,h)$, then $x\in \set{x\in \prod b}{x \sqrb (J,h)}$.   
\end{proof}

We below present connections between $\Rrm_b$ and measure zero.

\begin{lemma}\label{lem:b4}
 Let $b\in\omega^\omega$.  
    \begin{enumerate}[label=\rm(\arabic*)]
        \item\label{lem:b4:1} If $\sum_{k<\omega}\frac{1}{b(k)}<\infty$ then $\Crm_\Ewf \leqT \Rrm_b$.  In particular, $\bfrak_b^\eq \leq \non(\Ewf)$ and $\cov(\Ewf)\leq \dfrak_b^\eq$. 
        \item\label{lem:b4:2} If $\sum_{k<\omega}\frac{1}{b(k)} = \infty$. Then $\Crm_\Nwf\leqT\Rrm_b^\perp$. As a consequence, $\cov(\Nwf)\leq\bfrak_b^\eq$ and $\dfrak_b^\eq\leq\non(\Nwf)$. 
    \end{enumerate}
\end{lemma}
\begin{proof}
To prove~\ref{lem:b4:1}--\ref{lem:b4:2}, we work with $\Nwf(\prod b)$ instead of $\Nwf$ (see~\autoref{sec:intro}).

\ref{lem:b4:1}: Observe that \[\Lb_b([s]_b)=\Lb\bigg(\set{x\in\prod b}{ s\subseteq x}\bigg)=\prod_{i<|s|}\frac{1}{|b(i)|}\]
for any $s\in\Seq_{<\omega} b$. Let $F$ and $G$ be as in~\autoref{lem:b0}.  To complete the proof it suffices to prove  \[\Lb_b(\set{x\in \prod b}{x \sqrb (J,h)} )=0.\]

Recall
\[\set{x\in \prod b}{x \sqrb (J,h)} = \bigcup_{m<\omega}\bigcap_{n\geq m}\bigcup_{k\in J_n}A^{h(k)}_k.\]

Notice that $\Lb_b(A^\ell_k) = \frac{1}{b(k)}$, so we obtain
    \[\Lb_b\bigg(\set{x\in \prod b}{x \sqrb (J,h)}\bigg) \leq \lim_{m\to \infty}\prod_{n\geq m}\sum_{k\in J_n}\frac{1}{b(k)}.\]\
This limit above is $0$ because $\sum_{k<\omega}\frac{1}{b(k)}<\infty$.

\ref{lem:b4:2}: Since $\sum_{k<\omega}\frac{1}{b(k)} = \infty$, find $J\in\Ior$ such that $\sum_{k\in J_n}\frac{1}{b(k)} \geq n$ for all $n<\omega$. Observe that 
    \[\prod b\menos \set{x\in \prod b}{x \sqrb (J,h)} = \bigcap_{m<\omega}\bigcup_{n\geq m}\bigcap_{k\in J_n}\left(\prod b\menos A^{h(k)}_k\right),\]
    so
    \begin{align*}
        \Lb_b\left(\prod b\menos \set{x\in \prod b}{x \sqrb (J,h)}\right)& \leq \lim_{m\to\infty}\sum_{n\geq m}\prod_{k\in J_n}\left(1 - \frac{1}{b(k)}\right)\\
        &\leq \lim_{m\to\infty}\sum_{n\geq m} e^{-\sum_{k\in J_n}\frac{1}{b(k)}}.
    \end{align*}
Since $\sum_{k\in J_n}\frac{1}{b(k)} \geq n$, $\Lb_b\left(\prod b\menos \set{x\in \prod b}{x \sqrb (J,h)}\right)=0$.
 
Now we define the functions $\Psi_-$ and $\Psi_+$ as follows.   
 \begin{align*}
\Psi_-  & \colon \prod b  
                \longrightarrow \Ior\times\prod b \text{ and} \notag\\
\Psi_+ & \colon \prod b \longmapsto\Nwf(\prod b) \notag
\intertext{for each $h\in\prod b$ and $x\in\prod b$ by the assignments}
\Psi_-    & \colon x\longmapsto \big(J,x\big) \\
\Psi_+  & \colon h\longmapsto \prod b\smallsetminus \set{x\in \prod b}{x \sqrb (J,h)}
\end{align*}
It is not hard to see that for any $x\in\prod b$ and for any $h\in\prod b$, if $h\not\sqrb(J,x)$ then $x\in \prod b\smallsetminus \set{x\in \prod b}{x \sqrb (J,h)}$.  
\end{proof}

We introduce the following relational system for combinatorial purposes.

\begin{definition}\label{def:Ed}
  Let $b:=\seq{ b(n)}{ n<\omega}$ be a sequence of non-empty sets. Define the relational system $\Ed_b:=\la\prod b,\prod b,\neq^\infty\ra$ where $x=^\infty y$ means $x(n)=y(n)$ for infinitely many $n$. The relation $x\neq^\infty y$ means that \emph{$x$ and $y$ are eventually different}. Denote $\balc_{b,1}:=\bfrak(\Ed_b)$ and $\dalc_{b,1}:=\dfrak(\Ed_b)$.
\end{definition}

Recall the following characterization of the cardinal invariants associated with $\Mwf$. The one for $\add(\Mwf)$ is due to Miller~\cite{Miller}.

\begin{theorem}[{\cite[Sec.~3.3]{CM}}]\label{charM2}
    \[\add(\Mwf) = \min(\{\bfrak\}\cup\set{\dalc_{b,1}}{b\in \omega^\omega}) \text{ and } \cof(\Mwf) = \sup(\{\dfrak\}\cup\set{\balc_{b,1}}{b\in \omega^\omega})\]
\end{theorem}

Following,  we are establishing a connection between $\Rrm_b$ and $(\Ed_b^\perp,\Ior)$.

\begin{lemma}\label{dRbup}
    For $b\in\omega^\omega$, $\Rrm_b\leqT (\Ed_b^\perp,\Ior)$. As a consequence, $\dfrak_b^\eq\leq \max\{\balc_{b,1},\dfrak\}$ and $\min\{\dalc_{b,1},\bfrak\}\leq \bfrak_b^\eq$.
\end{lemma}
\begin{proof}
    Define $\Psi_-\colon \prod b \to \prod b \times \Ior^{\prod b}$ by $\Psi_-(x):= (x,F_x)$ where, for $y\in \Ior$, if $y =^\infty x$ then $F_x(y):= I^y_x \in \Ior$ is chosen such that $\forall k<\omega\, \exists i\in I^y_{x,k}\colon y(i) = x(i)$; otherwise, $F_x(y)$ can be anything (in $\Ior$).

    Define $\Psi_+\colon \prod b\times\Ior\to \Ior\times \prod b$ by $\Psi_+(y,J) = (J,y)$. We check that $(\Psi_-,\Psi_+)$ is a Tukey connection. Assume that $x,y\in\prod b$, $J\in\Ior$ and that $\Psi_-(x)\sqsubset_* (y,J)$, i.e.\ $x =^\infty y$ and $I^y_x \sqsubseteq J$. Since each $I^y_{x,k}$ contains a point where $x$ and $y$ coincide, $I^y_x \sqsubseteq J$ implies that, for all but finitely many $n<\omega$, $J_n$ contains a point where $x$ and $y$ coincide, which means that $x\sqrb (J,y) = \Psi_+(y,J)$.
\end{proof}

\autoref{charM2} and~\autoref{dRbup} together yield:

\begin{corollary}\label{supRbM}
    For all $b\in\omega^\omega$, $\dfrak_b^\eq \leq \cof(\Mwf)$.
\end{corollary}

Note that $\add(\Mwf) \leq \min\set{\bfrak_b^\eq}{b\in\omega^\omega}$ already follows from \autoref{bas:NA} and \eqref{la:addtch}.

We close this section with the proof of~\eqref{la:addtch}. Before, a natural question regarding~\eqref{la:addtch} that arises is 
\begin{question}
Does $\cov(\MAwf)=\sup\set{\dfrak_b^\eq}{b\in\omega^\omega}$ hold?
\end{question}

One negative answer to the prior question was given by the author along with Mej\'ia and Rivera-Madrid~\cite{CMR2}. Concretely, they proved the consistency of \[\sup\set{\dfrak_b^\eq}{b\in\omega^\omega}<\cov(\MAwf).\]  

We prove~\eqref{la:addtch} by using the subsequent two lemmas. 

\begin{lemma}\label{RsAMone}
Let $b\in\baire$. Then $\Rrm_{b}\leqT \Crm_\MAwf$. 
As a consequence, \[\non(\MAwf)\leq\min\set{\bfrak_b^\eq}{b\in\baire} \text{ and } \sup\set{\dfrak_b^\eq}{b\in\baire}\leq\cov(\MAwf).\]   
\end{lemma}
\begin{proof}
Given $b\in\baire$, thanks to \autoref{exa:a2} we may assume that there is some $I^b\in\Ior$ such that $b(n)=2^{|I_n^b|}$. 
Then, 
we can identify numbers ${<}b(n)$ with $0$-$1$ sequences of length $|I_n^b|$. 
We will find maps $\Psi_-\colon\prod b\to \cantor$  and  $\Psi_+\colon \MAwf\to \Ior\times\prod b$ such that, for any $ f\in\prod b$ and for any $X\in\MAwf$, $\Psi_-(f)\in X$ implies $f\sqrb \Psi_+(X)$.

Define $\Psi_-\colon\prod b\to \cantor$ as follows.
\begin{align*}
\Psi_-\colon\prod b & \rightarrow \cantor \\
       x & \mapsto  x_f^{I^b}={\underbrace{f(0)}_{\text{length }|I_0^b|}}^{\frown} \cdots\cdots{}^{\frown}{\underbrace{f(n)}_{\text{length } |I_n^b|}}^{\frown}\cdots
\end{align*} 

For $X\in \MAwf$, $X+B_{I^b}\in\Mwf$. Note that \[X+B_{I^b}=\bigcup_{x\in X}B_{x,I^b}.\]
Then, by~\autoref{ChM}, there are $y\in\cantor$ and $J\in\Ior$ such that 
\[\bigcup_{x\in X}B_{x,I^b}\subseteq B_{y, J}.\]
Let $h\in\prod b$ such that $y=x_{h}^{I^b}$ (recall that $b(n)=2^{|I^b_n|}$), so put $\Psi_+(X):=(J',h)$ where
\[k\in J'_n \text{ iff } \min J_n<\max I^b_k \leq \max J_n.\]

It remains to prove that, for any $ f\in\prod b$ and for any $X\in\MAwf$, $\Psi_-(f)\in X$ implies $ f\sqrb\Psi_+(X)$. Suppose that $x_{f}^{I_b}\in X$ and $\Psi_+(X) = (J',h)$. Then $B_{x_{f}^{I_b},I^b}\subseteq B_{x_{h}^{I^b},J}$. Hence, by using~\autoref{baseM}, 
\[\forall^\infty n\,\exists k\colon I_k^b\subseteq J_n\text{ and }x_f^{I^b}{\upharpoonright}I_k^b=x_{h}^{I^b}{\upharpoonright}I_k^b.\]
Since $I^b_k\subseteq J_n$ implies $k\in J'_n$, the equation above implies that $f\sqrb(J',h)$. 
\end{proof}

\begin{lemma}\label{prodRb}
    For any dominating family $D\subseteq\omega^\omega$, 
    $\Cbf_{\MAwf}\leqT \prod_{b\in D}\Rrm_b$. In particular, $\min_{b\in D}\bfrak_b^\eq \leq\non(\MAwf)$ and $\cov(\MAwf)\leq \prod_{b\in D}\dfrak_b^\eq$.
\end{lemma}
\begin{proof}
    Without loss of generality, we may assume that there is some $\Ior$-dominating family $D_0$, i.e.\ $\forall I\in \Ior\, \exists J\in D_0\colon I\sqsubseteq J$, such that for each $b\in D$ there is some $I\in D_0$ such that $b = 2^I$, i.e.\ $b(n) = 2^{I_n}$ for all $n<\omega$. 
    
    Define $\Psi_-\colon 2^\omega\to \prod_{I\in D_0} 2^I$ by $\Psi_-(x)(I):=\seq{x\frestr I_n}{n<\omega}$; and define $\Psi_+\colon \prod_{I\in D_0}\Ior\times 2^I\to \MAwf$ such that, for $z=\seq{(J^I,z^I)}{I\in\Ior}$,
    \[\Psi_+(z) := \set{x\in 2^\omega}{\forall I\in D_0\, \forall^\infty n<\omega\, \exists k\in J^I_k\colon x\frestr I_k = z^I(k)}.\]
    For each $I\in D_0$ let $I'_n:=\bigcup_{k\in J^I_n}I_k$ and $y^I\in 2^\omega$ the concatenation of all the $z^I(k)\in 2^{I_k}$ for $k<\omega$, i.e., $y^I\frestr I_k = z^I(k)$. Then $I':=\seq{I'_n}{n<\omega}\in \Ior$, $I\sqsubseteq I'$ and
    \[\forall^\infty n<\omega\, \exists k<\omega\colon I_k\subseteq I'_n \text{ and }x\frestr I_k = y\frestr I_k.\]
    Therefore, by \autoref{thm:a1}, $\Psi_+(z)\in \MAwf$. $(\Psi_-,\Psi_+)$ is clearly the required Tukey connection.
\end{proof}

\section{Consistent results}\label{sec:cons}

The main goal of this section is to establish~\autoref{thm:a6}, which is based on~\cite[Sec.~3]{CMR2}.

We now present one forcing notion closely related to $\bfrak_b^\eq$, that is to say, that increases $\bfrak_b^\eq$.

\begin{definition}[{\cite[Def.~3.20]{CMR2}}]\label{def_Pb}
Given $b\in\baire$, the poset $\Por_b$ is defined as follows:
A condition $p=(s,t,F)\in\Por_b$ if it fulfills the following:
\begin{itemize}
    \item $s\in\omega^{<\omega}$ is increasing with $s(0)>0$ (when $|s|>0$), 
    \item $t\in\Seq_{<\omega}(b):=\bigcup_{n<\omega}\prod_{i<n}b(i)$, and 
    \item $F\in[\prod b]^{<\aleph_0}$.
\end{itemize}
We order $\Por_b$ by setting $(s',t',F')\leq (s,t,F)$ iff $s\subseteq s'$, $t\subseteq t'$,
$F\subseteq F'$ and,
\[\forall f\in F\,\forall n \in|s'|\smallsetminus|s|\, \exists k\in[s'(n-1),s'(n))\colon f(k)=t'(k). \text{ (Here $s'(-1):=0$.)}\]
\end{definition}

\begin{fct}
  Let $b\in\baire$. Then $\Por_b$ is $\sigma$-centered.
\end{fct}
\begin{proof}
For $s\in\omega^{<\omega}$ increasing, and for $t\in\Seq_{<\omega}(b)$, set 
\[P_{s,t}:=\set{(s',t',F)\in\Por_b}{s'=s\textrm{\ and\ }  t'=t}\]
It is not hard to verify that $P_{s,t}$ is centered and $\bigcup_{s\in\omega^{<\omega},\, t\in\Seq_{<\omega}(b)}P_{s,t}=\Por_b$.    
\end{proof}

Let $G$ be a $\Por_b$-generic filter over $V$. In $V[G]$, define 
\[r_{\gen}:=\bigcup\set{s}{\exists t,F\colon (s,t,F) \in G} \text{ and } 
h_{\gen}:=\bigcup\set{t}{\exists s,F\colon (s,t,F) \in G}.\]
Then $(r_{\gen},h_{\gen})\in\baire\times\prod b$ and, for every $f\in \prod b\cap V$, and for all but finitely many $n\in\omega$ there is some $k\in[r_{\gen}(n),r_{\gen}(n+1)]$ such that $f(k)=h_{\gen}(k)$. We can identify the generic real with $(J_{\gen},h_{\gen})\in\Ior\times\prod b$ where $J_{\gen,n} :=[r_\gen(n-1),r_\gen(n))$, which satisfies that, for every $f\in \prod b\cap V$, $f  \sqrb (J_{\gen},h_{\gen})$.

\begin{definition}[{\cite{mejiavert,BCM}}]\label{def:Fr-l}
    Let $F\subseteq\pts(\omega)$ be a filter. We assume that all filters are \emph{free}, i.e.\ they contain the \emph{Frechet filter} $\Fr:=\set{\omega\menos a}{a\in[\omega]^{<\aleph_0}}$. A set $a\subseteq \omega$ is \emph{$F$-positive} if it intersects every member of $F$. Denote by $F^+$ the collection of $F$-positive sets. 

    Given a poset $\Por$ and $Q\subseteq \Por$, $Q$ is $F$-linked if, for any $\la p_n\colon n<\omega\ra\in Q^\omega$, there is some $q\in \Por$ such that
    \[q\Vdash \{n<\omega\colon p_n\in \dot G\} \in F^+,\text{ i.e.\ it intersects every member of $F$.}\]
    Note that, in the case $F=\Fr$, the previous equation is ``$q\Vdash \{n<\omega\colon p_n\in \dot G\}$ is infinite".

    We say that $Q$ is \emph{uf-linked (ultrafilter-linked)} if it is $F$-linked for any filter $F$ on $\omega$ containing the \emph{Frechet filter} $\Fr$.
    
    For an infinite cardinal $\mu$, $\Por$ is \emph{$\mu$-$F$-linked} if $\Por = \bigcup_{\alpha<\mu}Q_\alpha$ for some $F$-linked $Q_\alpha$ ($\alpha<\mu$). When these $Q_\alpha$ are uf-linked, we say that $\Por$ is \emph{$\mu$-uf-linked}.
\end{definition}

Note that if $F\subseteq F'$ are filters on $\omega$, then $\sigma$-$\rm uf$-linked${}\Rightarrow{}\sigma$-$F'$-linked${}\Rightarrow{}\sigma$-$F$-linked${}\Rightarrow{}\sigma$-$\Fr$-linked. For ccc posets, however, we have:

\begin{lemma}[{\cite[Lem~5.5]{mejiavert}}]\label{quasiuf}
If $\Por$ is ccc then any subset of $\Por$ is uf-linked iff it is Fr-linked. 
\end{lemma}

Below are presented a few well-known and basic instances of $\sigma$-uf-linked posets.

\begin{example}\label{exa:b3}
    \ 
    \begin{enumerate}[label = \normalfont (\arabic*)]
        \item\label{exa:b3:1} Let $\Por$ be a poset and $Q\subseteq\Por$.  Note that a sequence $\seq{ p_n}{n<\omega}$ in $Q$ witnesses that $Q$ is \underline{not} $\Fr$-linked iff the set $\set{q\in\Por}{\forall^\infty n<\omega\colon q\perp p_n}$ is dense. 
       
        \item\label{exa:b3:2} Any singleton is uf-linked. Hence, any poset $\Por$ is $|\Por|$-uf-linked. In particular, Cohen forcing is $\sigma$-uf-linked.

        \item\label{exa:b3:3} Random forcing $\Bor$ is $\sigma$-uf-linked \cite{mejiavert}.

        \item\label{exa:b3:4} The forcing eventually different real forcing $\Eor$ (see \cite{mejiavert}) is $\sigma$-uf-linked. 
         This poset satisfies a stronger property see~\autoref{exm:var}~\ref{exm:var:2}.
    \end{enumerate}
\end{example}

The upcoming lemma indicates that $\sigma$-$\Fr$-linked poset does not add dominating reals. 

\begin{lemma}[{\cite{mejiavert}}]\label{lem:d11}
    Any $\mu$-$\Fr$-linked poset is $\mu^+$-$\omega^\omega$-good.
\end{lemma}

We now focus on reviewing one linkedness property stronger than ultrafilter linkedness.

\begin{definition}[{\cite{GMS,BCM,CMR2}}]\label{defuflim}
    Given a (non-principal) ultrafilter $D$ on $\omega$ and $Q\subseteq\Por$, say that $Q$ is \emph{$D$-$\lim$-linked} if there are a $\Por$-name $\dot D'$ of an ultrafilter on $\omega$ extending $D$ and a map $\lim^D\colon Q^\omega\to \Por$ such that, whenever $\bar p = \la p_n\colon n<\omega\ra\in Q^\omega$,
    \[{\lim}^D\, \bar p \Vdash \{n<\omega\colon p_n\in \dot G\} \in \dot D'.\]
    A set $Q\subseteq \Por$ has \emph{uf}-$\lim$-linked if it is $D$-\textrm{lim}-linked for any ultrafilter $D$. 

In addition, for an infinite cardinal $\theta$, the poset $\Por$ is~\emph{uniformly $\mu$-$D$-$\lim$-linked} if $\Por = \bigcup_{\alpha<\theta}Q_\alpha$ where each $Q_\alpha$ is $D$-$\lim$-linked and the $\Por$-name $\dot D'$ above mentioned only depends on $D$ (and not on $Q_\alpha$, although we have different limits for each $Q_\alpha$. When these $Q_\alpha$ are \emph{uf}-$\lim$-linked, we say that $\Por$ is~\emph{uniformly $\mu$-uf-$\lim$-linked}
\end{definition}

\begin{remark}
Any uf-$\lim$-linked set $Q\subseteq \Por$ is clearly uf-linked, which implies that it is $\Fr$-linked.
\end{remark}

\begin{example}\label{exm:var}
\ 
    \begin{enumerate}[label=\rm(\arabic*)]
        \item Any singleton is uf-$\lim$-linked. As a consequence, any poset $\Por$ is uniformly $|\Por|$-uf-$\lim$-linked, witnessed by its singletons.
       \item\label{exm:var:2}  $\Eor$ is uniformly $\sigma$-uf-$\lim$-linked (see~\cite{GMS}, see also~\cite{Miller}).

       \item $\Bor$ is not $\sigma$-uf-$\lim$-linked (see~\cite[Rem.~3.10]{BCM}).
    \end{enumerate}
\end{example}

Next, we show another example of uniformly $\sigma$-uf-$\lim$-linked.

\begin{lemma}[{\cite[Thm.~3.21]{CMR2}}]\label{lem:b11}
Let $b\in\baire$. Then $\Por_b$ is uniformly $\sigma$-uf-$\lim$-linked.
\end{lemma}
\begin{proof}
For $s\in\omega^{<\omega}$, $t\in\Seq_{<\omega}(b)$ and $m<\omega$
\[P_{s,t,m}:=P_b(s,t,m)=\set{(s',t',F)\in\Por_b}{s'=s,\  t'=t \text{ and }|F|\leq m}.\]
For an ultrafilter $D$ on $\omega$,  and $\bar{p}=\seq{p_n}{n\in\omega}\in P_{s,t,m}$, we show how to define $\lim^D\bar p$. Let $p_n=(s,t,F_n) \in P_{s,t,m}$. Considering the lexicographic order $\lhd$ of $\prod b$, and let $\set{x_{n,k}}{k<m_n}$ be a $\lhd$-increasing enumeration of $F_n$ where $m_n \leq m$. Next find an unique  $m_*\leq m$ such that $A:=\set{n\in\omega}{m_n=m_*}\in D$. For each $k<m_*$, define $x_k:=\lim_n^D x_{n,k}$ in $\prod b$ where $x_k(i)$ is the unique member of $b(i)$ such that $\set{n\in A}{x_{n,k}(i) = x_k(i)} \in D$ (this coincides with the topological $D$-limit). Therefore, we can think of $F:=\set{x_k}{k<m_*}$ as the $D$-limit of $\seq{F_n}{n<\omega}$, so we define $\lim^D \bar p:=(s,t,F)$. Note that $\lim^D \bar p \in P_{s,t,m}$.

The sequence $\la P_{s,t,m}\colon s\in\omega^{<\omega}, t\in\Seq_{<\omega}(b), m<\omega\ra$ witnesses that $\Por_b$ is uniformly $\sigma$-$D$-$\lim$-linked for any ultrafilter $D$ on $\omega$. This is a consequence of the following claim:
\begin{clm}[{\cite[Claim.~3.22]{CMR2}}]
The set
    \[D\cup\bigcup_{s,m}\lset{\{n<\omega\colon p_n\in G\}}{\bar p\in P_{s,t,m}^\omega\cap V,\ {\lim}^D\, \bar p\in G}\]
    has the finite intersection property whenever $G$ is $\Por$-generic over $V$. \qedhere
\end{clm}
\end{proof}

We below review briefly the preservation theory of unbounded families
presented in~\cite[Sect.~4]{CM}. This a generalization of Judah and Shelah's~\cite{JS} and
Brendle's~\cite{Br} preservation theory.

\begin{definition}\label{DefRelSys}
Let $\Rrm=\la X,Y,\sqsubset\ra$ be a relational system and let $\theta$ be a cardinal.  
\begin{enumerate}[label=\rm(\arabic*)]
  \item For a set $M$,
  \begin{enumerate}[label=\rm(\roman*)]
      \item An object $y\in Y$ is \textit{$\Rrm$-dominating over $M$} if $x\sqsubset y$ for all $x\in X\cap M$. 
      
      \item An object $x\in X$ is \textit{$\Rrm$-unbounded over $M$} if it $\Rrm^\perp$-dominating over $M$, that is, $x\not\sqsubset y$ for all $y\in Y\cap M$. 
  \end{enumerate}
  
  
  \item A family $\set{x_i}{i\in I}\subseteq X$ is \emph{strongly $\theta$-$\Rrm$-unbounded} if 
  $|I|\geq\theta$ and, for any $y\in Y$, $|\set{i\in I }{x_i\sqsubset y}|<\theta$.
\end{enumerate}
\end{definition}

We look at the following type of well-defined relational systems.

\begin{definition}\label{def:Prs}
Say that $\Rrm=\langle X,Y,\sqsubset\rangle$ is a \textit{Polish relational system (Prs)} if
\begin{enumerate}[label=\rm(\arabic*)]
\item\label{def:Prs1} $X$ is a Perfect Polish space,
\item\label{def:Prs2} $Y$ is a non-empty analytic subspace of some Polish $Z$, and
\item\label{def:Prs3} $\sqsubset=\bigcup_{n<\omega}\sqsubset_{n}$ where $\seq{{\sqsubset_{n}}}{n\in\omega}$ is some increasing sequence of closed subsets of $X\times Z$ such that, for any $n<\omega$ and for any $y\in Y$,
$({\sqsubset_{n}})^{y}=\set{x\in X}{x\sqsubset_{n}y }$ is closed nowhere dense.
\end{enumerate}
\end{definition}

\begin{remark}\label{Prsremark}
By~\autoref{def:Prs}~\ref{def:Prs3}, $\la X,\Mwf(X),\in\ra$ is Tukey below $\Rrm$ where $\Mwf(X)$ denotes the $\sigma$-ideal of meager subsets of $X$. Therefore, $\bfrak(\Rrm)\leq \non(\Mwf)$ and $\cov(\Mwf)\leq\dfrak(\Rrm)$.
\end{remark}

For the rest of this section, fix a Prs $\Rrm=\langle X,Y,\sqsubset\rangle$ and an infinite cardinal $\theta$. 

\begin{definition}[Judah and Shelah {\cite{JS}}, Brendle~{\cite{Br}}]\label{def:good}
A poset $\Por$ is \textit{$\theta$-$\Rrm$-good} if, for any $\Por$-name $\dot{h}$ for a member of $Y$, there is a non-empty set $H\subseteq Y$ (in the ground model) of size ${<}\theta$ such that, for any $x\in X$, if $x$ is $\Rrm$-unbounded over  $H$ then $\Vdash x\not\sqsubset \dot{h}$.

We say that $\Por$ is \textit{$\Rrm$-good} if it is $\aleph_1$-$\Rrm$-good.
\end{definition}

The previous is a standard property associated with preserving $\bfrak(\Rrm)$ small and $\dfrak(\Rrm)$ large after forcing extensions.

\begin{remark}
Notice that $\theta<\theta_0$
implies that any $\theta$-$\Rrm$-good poset is $\theta_0$-$\Rrm$-good. Also, if $\Por \lessdot\Qor$ and $\Qor$ is $\theta$-$\Rrm$-good, then $\Por$ is $\theta$-$\Rrm$-good.
\end{remark}

\begin{lemma}[{\cite[Lemma~2.7]{CM}}]\label{lem:c2}
Assume that $\theta$ is a regular cardinal. Then any poset of size ${<}\theta$
is $\theta$-$\Rrm$-good. In particular, Cohen forcing $\Cor$ is $\Rrm$-good.
\end{lemma}

We now present the instances of Prs and the corresponding good posets that we use in our applications.
\newcommand{\sqce}{\sqsubset^{\star}}
\newcommand{\Ce}{\mathsf{Ce}}
\begin{example}\label{exa:c2}
\mbox{} 
\begin{enumerate}[label=\normalfont(\arabic*)]

    \item\label{exa:c2:1}  Define $\Omega_n:=\set{a\in [2^{<\omega}]^{<\aleph_0}}{\Lb(\bigcup_{s\in a}[s])\leq 2^{-n}}$ (endowed with the discrete topology) and put $\Omega:=\prod_{n<\omega}\Omega_n$ with the product topology, which is a perfect Polish space. For every $x\in \Omega$ denote 
     \[N_{x}:=\bigcap_{n<\omega}\bigcup_{m\geq n}\bigcup_{s\in x(m)}[s],\] which is clearly a Borel null set in $2^{\omega}$.
       
    Define the Prs $\Cn:=\la \Omega, 2^\omega, \sqsubset^{\rm n}\ra$ where $x\sqsubset^{\rm n} z$ iff $z\notin N_{x}$. Recall that any null set in $2^\omega$ is a subset of $N_{x}$ for some $x\in \Omega$, so $\Cn$ and $\Cbf_\Nwf^\perp$ are Tukey-Galois equivalent. Hence, $\bfrak(\Cn)=\cov(\Nwf)$ and $\dfrak(\Cn)=\non(\Nwf)$.
    
 Any $\mu$-centered poset is $\mu^+$-$\Cn$-good (\cite{Br}). In particular, $\sigma$-centered posets are $\Cn$-good.
    
    \item\label{exa:c2:2}
    The relational system $\Ed_b$ is Polish when $b=\seq{b(n)}{n<\omega}$ is a sequence of non-empty countable sets such that $|b(n)|\geq 2$ for infinitely many $n$.
    Consider $\Ed:=\la\baire,\baire,\neq^\infty\ra$. 
    By~\cite[Thm.~2.4.1 \& Thm.~2.4.7]{BJ} (see also~\cite[Thm.~5.3]{CMlocalc}),  $\bfrak(\Ed)=\non(\Mwf)$ and $\dfrak(\Ed)=\cov(\Mwf)$.
    
    \item\label{exa:c2:3} The relational system $\baire = \la\baire,\baire,\leq^{*}\ra$ is Polish.
     Any $\mu$-$\Fr$-linked poset is $\mu^+$-$\baire$-good (see~\autoref{lem:d11}).
    
    \item\label{exa:c2:4}  For each $k<\omega$, let $\id^k:\omega\to\omega$ such that $\id^k(i)=i^k$ for all $i<\omega$ and $\Hwf:=\set{\id^{k+1}}{k<\omega}$. Let $\Lc^*:=\la\baire, \Scal(\omega, \Hwf), \in^*\ra$ be the Polish relational system where \[\Swf(\omega, \Hwf):=\set{\varphi\colon \omega\to[\omega]^{<\aleph_0}}{\exists{h\in\Hwf}\, \forall{i<\omega}\colon|\varphi(i)|\leq h(i)},\]
     and recall that $x\in^*\varphi$ iff $\forall^\infty n\colon x(n)\in\varphi(n)$. As a consequence of~\cite[Thm.~2.3.9]{BJ} (see also~\cite[Thm.~4.2]{CMlocalc}), $\bfrak(\Lc^*)=\add(\Nwf)$ and $\dfrak(\Lc^*)=\cof(\Nwf)$.

     Any $\mu$-centered poset is $\mu^+$-$\Lc^*$-good (see~\cite{Br,JS}) so, in particular, $\sigma$-centered posets are $\Lc^*$-good. Besides,  Kamburelis~\cite{Ka} showed that any Boolean algebra with a strictly positive finitely additive measure is $\Lc^*$-good (in particular, any subalgebra of random forcing).

     \item For $b\in\omega^\omega$, $\Rrm_b$ is a Polish relational system when $b\geq^*2$ (see~\autoref{exa:a2}).

     \item\label{exa:c2:5} Let $\Mbf := \la 2^\omega,\Ior\times 2^\omega, \sqsubm \ra$ where
     \[x \sqsubm (I,y) \text{ iff }\forall^\infty n\colon x\frestr I_n \neq y\frestr I_n.\]
     This is a Polish relational system and $\Mbf\eqT \Cbf_\Mwf$ (by \autoref{ChM}).

     Note that, whenever $M$ is a transitive model of $\thzfc$, $c\in 2^\omega$ is a Cohen real over $M$ iff $c$ is $\Mbf$-unbounded over $M$.

   \item    Define the relational system $\Ce=\la \cantor,\mathsf{NE},\sqce\ra$ where $\mathsf{NE}$ is the collection of sequences $\bar T = \langle T_n \colon n<\omega \rangle$ such that each $T_n$ is a subtree of ${}^{<\omega}2$ (not necessarily well-pruned), $T_n\subseteq T_{n+1}$ and $\Lb([ T_n]) =0$, i.e.\ $\displaystyle \lim_{n\to\infty}\frac{|T\cap {}^n2|}{2^n}=0$, and $x\sqce \bar T$ iff $x\in [T_n]$ for some $n<\omega$. 
\end{enumerate}
\end{example}




Good posets are preserved along FS iterations as follows.

\begin{lemma}[{\cite[Sec.~4]{BCM2}}]\label{lem:c3}
Let $\seq{ \Por_\xi,\Qnm_\xi}{\xi<\pi}$ be a FS iteration such that, for $\xi<\pi$, $\Por_\xi$ forces that $\Qnm_\xi$ is a non-trivial $\theta$-cc $\theta$-$\Rrm$-good poset. 
Let $\set{\gamma_\alpha}{\alpha<\delta}$ be an increasing enumeration of $0$ and all limit ordinals smaller than $\pi$ (note that $\gamma_\alpha=\omega\alpha$), and for $\alpha<\delta$ let $\dot c_\alpha$ be a $\Por_{\gamma_{\alpha+1}}$-name of a Cohen real in $X$ over $V_{\gamma_\alpha}$. 

Then $\Por_\pi$ is $\theta$-$\Rrm$-good. Moreover, 
if $\pi\geq\theta$ then $\Cbf_{[\pi]^{<\theta}}\leqT\Rrm$, $\bfrak(\Rrm)\leq\theta$ and $|\pi|\leq\dfrak(\Rrm)$.
\end{lemma}

To force a lower bound of $\bfrak(\Rrm)$, we use:

\begin{lemma}[{\cite[Thm.~2.12]{CM22}}]\label{lem:c4}
Let $\Rrm=\la X,Y,\sqsubset\ra$ be a Polish relational system, $\theta$ an uncountable regular cardinal, and let  $\Por_\pi=\seq{\Por_\xi,\Qnm_\xi}{\xi<\pi}$ be a FS iteration of $\theta$-cc posets with $\cf(\pi)\geq\theta$. Assume that, for all $\xi<\pi$ and any $A\in[X]^{<\theta}\cap V_\xi$, there is some $\eta\geq\xi$ such that $\Qnm_\eta$ adds an $\Rrm$-dominating real over $A$. Then $\Por_\pi$ forces $\theta\leq\bfrak(\Rrm)$, i.e.\ $\Rrm\leqT\Cbf_{[X]^{<\theta}}$.
\end{lemma}

The following results illustrates the effect of adding cofinally many $\Rrm$-dominating reals along a FS iteration.

\begin{lemma}[{\cite[Lem.~2.9]{CM22}}]\label{lem:c5}
Let $\Rrm$ be a definable relational system of the reals, and let $\lambda$ be a limit ordinal of uncountable cofinality.
If $\Por_\lambda=\la\Por_\xi,\Qnm_\xi:\,  \xi<\lambda\ra$ is a FS iteration of $\cf(\lambda)$-cc posets that adds $\Rrm$-dominating reals cofinally often, then $\Por_\lambda$ forces $\Rrm\leqT\lambda$. 

In addition, if $\Rrm$ is a Prs and all iterands are non-trivial, then $\Por_\lambda$ forces $\Rrm\eqT\Mbf\eqT\lambda$.
In particular, $\Por_\lambda$ forces $\bfrak(\Rrm)=\dfrak(\Rrm)=\non(\Mwf)=\cov(\Mwf) =\cf(\lambda)$.
\end{lemma}

Next, we illustrate the effect of iterating $\Por_b$ on Cicho\'n's diagram.

\begin{theorem}\label{thm:d9}
Let $\pi$ be an ordinal of uncountable cofinality such that $|\pi|^{\aleph_0}=|\pi|$. The FS iteration of $\Por$ of length $\pi$ (i.e. the FS iteration $\la\Por_\alpha,\Qnm_\alpha:\, \alpha<\pi\ra$ where each $\Qnm_\alpha$ is a $\Por_\alpha$-name of $\Por_b$) forces $\cfrak=|\pi|$, $ \Crm_\Mwf\eqT\pi$ and $ \Crm_\Nwf^\perp\eqT\omega^\omega\eqT\Cbf_{[\R]^{<\aleph_1}}$. In particular, it forces $\cov(\Nwf)=\bfrak=\aleph_1$, $\bfrak_b^\eq=\non(\Mwf)=\cov(\Mwf)=\dfrak_b^\eq=\cf(\pi)$ and $\non(\Nwf)=\dfrak=\cfrak=|\pi|$ (see~\autoref{Figthm:d9}).
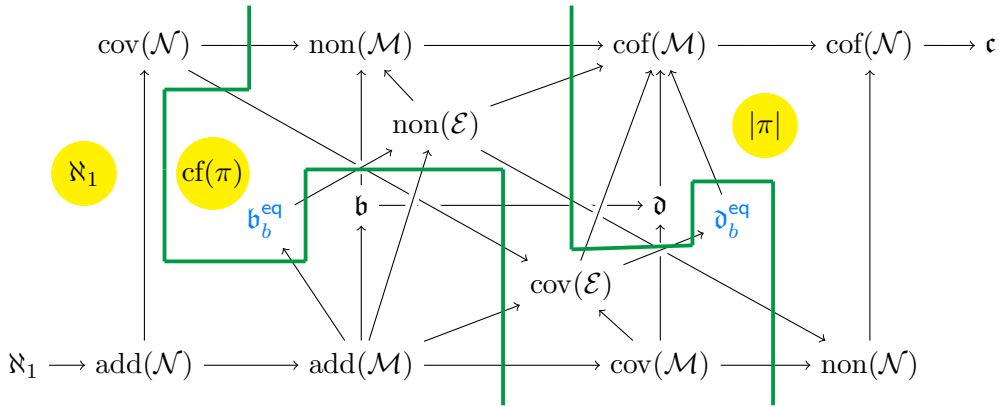
\begin{figure}[ht!]
\centering
\begin{tikzpicture}[scale=1.06]
\small{
\node (aleph1) at (-1,3) {$\aleph_1$};
\node (addn) at (0.5,3){$\add(\Nwf)$};
\node (covn) at (0.5,7){$\cov(\Nwf)$};
\node (nonn) at (9.5,3) {$\non(\Nwf)$} ;
\node (cfn) at (9.5,7) {$\cof(\Nwf)$} ;
\node (addm) at (3.19,3) {$\add(\Mwf)$} ;
\node (covm) at (6.9,3) {$\cov(\Mwf)$} ;
\node (nonm) at (3.19,7) {$\non(\Mwf)$} ;
\node (cfm) at (6.9,7) {$\cof(\Mwf)$} ;
\node (b) at (3.19,5) {$\bfrak$};
\node (d) at (6.9,5) {$\dfrak$};
\node (c) at (11,7) {$\cfrak$};
\node (e) at (2,4.8) {\dodger{$\bfrak_b^\eq$}};
\node (deq) at (7.8,4.8) {\dodger{$\dfrak_b^\eq$}};
\node (none) at (4.12,6) {$\non(\Ewf)$};
\node (cove) at (5.8,4) {$\cov(\Ewf)$};
\draw (aleph1) edge[->] (addn)
      (addn) edge[->] (covn)
      (covn) edge [->] (nonm)
      (nonm)edge [->] (cfm)
      (cfm)edge [->] (cfn)
      (cfn) edge[->] (c);

\draw
   (addn) edge [->]  (addm)
   (addm) edge [->]  (covm)
   (covm) edge [->]  (nonn)
   (nonn) edge [->]  (cfn);
\draw (addm) edge [->] (b)
      (b)  edge [->] (nonm);
\draw (covm) edge [->] (d)
      (d)  edge[->] (cfm);
\draw (b) edge [->] (d);

\draw   (none) edge [->] (nonm)
        (none) edge [->] (cfm)
        (addm) edge [->] (cove);
      
\draw (none) edge [line width=.15cm,white,-] (nonn)
      (none) edge [->] (nonn);
      
\draw (cove) edge [line width=.15cm,white,-] (covn)
      (cove) edge [<-] (covm)
      (cove) edge [<-] (covn);

\draw (addm) edge [line width=.15cm,white,-] (none)
      (addm) edge [->] (none); 

\draw (cove) edge [line width=.15cm,white,-] (cfm)
      (cove) edge [->] (cfm);

\draw (e) edge [line width=.15cm,white,-] (none)
      (e) edge [->] (none);   

\draw (e) edge [line width=.15cm,white,-] (addm)
      (e) edge [<-] (addm);  


\draw (deq) edge [line width=.15cm,white,-] (cove)
      (deq) edge [<-] (cove);  

\draw (deq) edge [line width=.15cm,white,-] (cfm)
      (deq) edge [->] (cfm);  

\draw[color=sug,line width=.05cm] (0.75,4.3)--(2.5,4.3); 
\draw[color=sug,line width=.05cm] (2.5,5.45)--(4.95,5.45);  
\draw[color=sug,line width=.05cm] (0.75,6.45)--(1.8,6.45);   
\draw[color=sug,line width=.05cm] (5.8,4.45)--(5.8,7.5); 

\draw[color=sug,line width=.05cm] (7.3,4.5)--(7.3,5.3); 

\draw[color=sug,line width=.05cm] (8.3,2.5)--(8.3,5.3); 
\draw[color=sug,line width=.05cm] (4.95,2.5)--(4.95,5.45);
\draw[color=sug,line width=.05cm] (1.8,6.45)--(1.8,7.5);
\draw[color=sug,line width=.05cm] (0.75,5.45)--(0.75,6.45); 
\draw[color=sug,line width=.05cm] (0.75,4.3)--(0.75,5.45);
\draw[color=sug,line width=.05cm] (2.5,4.3)--(2.5,5.45);

\draw[color=sug,line width=.05cm] (5.8,4.45)--(7.3,4.5);
\draw[color=sug,line width=.05cm] (7.3,5.3)--(8.3,5.3);

\draw[circle, fill=yellow,color=yellow] (1.35,5.4) circle (0.45);
\draw[circle, fill=yellow,color=yellow] (-0.25,5.4) circle (0.4);
\draw[circle, fill=yellow,color=yellow] (8.2,6) circle (0.4);
\node at (8.2,6) {$|\pi|$};
\node at (1.35,5.4) {$\cf(\pi)$};
\node at (-0.25,5.4) {$\aleph_1$};
}
\end{tikzpicture}
\caption{Cicho\'n's diagram after adding $\pi$-many generic reals with $\Por_b$, where $\pi$ has uncountable cofinality and $|\pi|^{\aleph_0}=|\pi|$.}
\label{Figthm:d9}
\end{figure}
\end{theorem}


The proof of the above theorem follows from~\autoref{thm:a6}, so we will now prove~\autoref{thm:a6}.

\begin{theorem}\label{thm:c2}
    Let $\aleph_1\leq \lambda_1\leq \lambda_2\leq \lambda_3 \leq \lambda_4$ be regular cardinals, and assume $\lambda_5$ is a cardinal such that 
  $\lambda_5>\lambda_4$ and $\lambda_5= \lambda_5^{\aleph_0}$ and $\cf([\lambda_5]^{<\lambda_i}) = \lambda_5$ for $i=1,\ldots,4$. Then, we can construct a FS iteration of length $\lambda$ of ccc posets forcing  $\cfrak = \lambda _5$, $\Lc^*\eqT\Cbf_{[\lambda _5]^{<{\lambda _1}}}$,  $ \Crm_{\Nwf}^{\perp}\eqT\Cbf_{[\lambda _5]^{<{\lambda _3}}}$, $\baire\eqT\Cbf_{[\lambda _5]^{<{\lambda _4}}}$, and $\Rrm_b\eqT\Mbf\eqT\lambda_4$. In particular, it forced~\autoref{Figthm:a6}.
\end{theorem}
\begin{proof}
Make a FS iteration $\Por=\la\Por_\xi,\Qnm_\xi:\,  \xi<\lambda\ra$ of length $\lambda:=\lambda_5\lambda_4$ (ordinal product) as follows. Fix a partition $\la C_i:\,  1\leq i\leq 3\ra$ of $\lambda_5\smallsetminus\{0\}$ where each set has size $\lambda_5$. For each $\rho<\lambda_4$ denote $\eta_\rho:=\lambda_5\rho$. We define the iteration at each $\xi=\eta_\rho+\varepsilon$ for $\rho<\lambda_4$ and $\varepsilon<\lambda_5$ as follows:

\[\Qnm_\xi:=\left\{\begin{array}{ll}
        \Por_b   & \text{if $\varepsilon=0$,}\\
        \LOCor^{\dot N_\xi}  & \text{if $\varepsilon\in C_1$,} \\
        \Bor^{\dot N_\xi} & \text{if $\varepsilon\in C_2$,} \\ 
        \Dor^{\dot N_\xi} & \text{if $\varepsilon\in C_3$,}
    \end{array}\right.\]
where $\dot N_\xi$ is a $\Por_{\xi}$-name of a transitive model of $\thzfc$ of size ${<}\lambda_i$ when $\varepsilon\in C_i$. 

Additionally, by a book-keeping argument, we make sure that all such models $N_\xi$ are constructed such that, for any $\rho<\lambda_4$:
\begin{enumerate}[label=\rm$(\bullet_\arabic*)$]
    \item\label{thm:c2:1} if $A\in V_{\eta_\rho}$ is a subset of $\omega^\omega$ of size ${<}\lambda_1$, then there is some $\varepsilon\in C_1$ such that  $A\subseteq N_{\eta_\rho+\varepsilon}$;
    
     \item if $A\in V_{\eta_\rho}$ is a subset of $\Omega$ of size ${<}\lambda_2$, then there is some $\varepsilon\in C_2$ such that  $A\subseteq N_{\eta_\rho+\varepsilon}$; and

    \item if $A\in V_{\eta_\rho}$ is a subset of $\omega^\omega$ of size ${<}\lambda_3$, then there is some $\varepsilon\in C_3$ such that $A\subseteq N_{\eta_\rho+\varepsilon}$.   
\end{enumerate}

We prove that $\Por$ is as required. Clearly, $\Por$ forces $\cfrak=\lambda_5$.

Notice first that all iterands are $\lambda_1$-$\Lc^*$-good (see \autoref{lem:c2} and \autoref{exa:c2}~\ref{exa:c2:4}), $\lambda_2$-$\Crm_{\Nwf}^{\perp}$-good (see \autoref{lem:c2} and \autoref{exa:c2}~\ref{exa:c2:1}) and $\lambda_3$-$\baire$-good (see \autoref{lem:c2}, \autoref{lem:b11}, and \autoref{lem:d11}), so by~\autoref{lem:c3} we obtain~$\Por$ forces $\Cbf_{[\lambda _5]^{<{\lambda _1}}}\leqT\Lc^*$, $\Cbf_{[\lambda _5]^{<{\lambda _1}}}\leqT\Crm_{\Nwf}^{\perp}$ and $\Cbf_{[\lambda _5]^{<{\lambda _1}}}\leqT\baire$. On the other hand, by using~\ref{thm:c2:1} and~\autoref{lem:c4}, $\Por$ forces $\Lc^*\leqT\Cbf_{[\lambda _5]^{<{\lambda _1}}}$. 

In a similar way to the previous argument, $\Por$ forces $\Crm_{\Nwf}^{\perp}\leqT\Cbf_{[\lambda _5]^{<{\lambda _3}}}$, $\baire\leqT\Cbf_{[\lambda _5]^{<{\lambda _4}}}$. 

Finally, since $\cf(\lambda)=\lambda_4$, by \autoref{lem:c5} $\Por$ forces $\Rrm_b\eqT\Mbf\eqT\lambda_4$ because $\Por_b$ adds $\Rrm_b$-dominating reals. 
\end{proof}

According to the preceding theorem, $\bfrak_b^\eq>\cov(\Nwf)$ is consistent. However, what about $\cov(\Nwf)>\bfrak_b^\eq$? We then give a positive answer to this question.

A notion proceeding ultrafilter-limits, which is more powerful, is finitely additive measures
(fams)-limits introduced implicitly in the proof of the consistency of $\cf(\cov(\Nwf )) = \omega$ by
Shelah~\cite{ShCov} and was formalized in~\cite{KST}. Recently, the author refined this in general setting along with Mej\'ia, and Uribe-Zapata~\cite{CMU}.

\begin{definition}[{\cite{CMU}}]
    Let $\Por$ be a poset and let $\Xi\colon \pts(\omega) \to [0,1]$ be a fam (with $\Xi(\omega)=1$ and $\Xi(\{n\})=0$ for all $n<\omega$), $I = \la I_n\colon n<\omega\ra$ be a partition of $\omega$ into finite sets, and $\varepsilon>0$.
    \begin{enumerate}[label = \normalfont (\arabic*)]

        \item\label{faml1} A set $Q\subseteq\Por$ is \emph{$(\Xi,I,\varepsilon)$-linked} if there is a function $\lim\colon Q^\omega\to \Por$ and a $\Por$-name $\dot \Xi'$ of a fam on $\pts(\omega)$ extending $\Xi$ such that, for any $\bar p = \seq{ p_\ell}{\ell<\omega} \in Q^\omega$,
        \[\lim \bar p \Vdash \int_\omega \frac{|\set{\ell \in I_k}{p_\ell \in \dot G}|}{|I_k|}d\dot \Xi' \geq 1-\varepsilon.\]

        \item\label{faml2} The poset $\Por$ is \emph{$\mu$-$\mathsf{FAM}$-linked}, witnessed by $\la Q_{\alpha,\varepsilon}\colon \alpha<\mu,\ \varepsilon\in(0,1)\cap \Q\ra$, if:
        \begin{enumerate}[label = \rm (\roman*)]
            \item Each $Q_{\alpha,\varepsilon}$ is $(\Xi,I,\varepsilon)$-linked for any $\Xi$ and $I$.
            \item For $\varepsilon\in(0,1)\cap \Q$, $\bigcup_{\alpha<\omega} Q_{\alpha,\varepsilon}$ is dense in $\Por$.
        \end{enumerate}

        \item\label{faml3} The poset $\Por$ is \emph{uniformly $\mu$-$\mathsf{FAM}$-linked} if there is some $\la Q_{\alpha,\varepsilon}\colon \alpha<\mu,\ \varepsilon\in(0,1)\cap \Q\ra$ as above, such that in~\ref{faml1} the name $\dot \Xi'$ only depends on $\Xi$ (and not on any $Q_{\alpha,\varepsilon}$).
    \end{enumerate}
\end{definition}

\begin{example}\label{exa:b2}
    \ 
    \begin{enumerate}[label = \normalfont (\arabic*)]
        \item\label{exa:b2:0} Any singleton is $(\Xi,I,\varepsilon)$-linked. Hence, any poset $\Por$ is uniformly $|\Por|$-$\mathsf{FAM}$-linked. In particular, Cohen forcing is uniformly $\sigma$-$\mathsf{FAM}$-linked.

        \item\label{exa:b2:1} Shelah~\cite{ShCov} proved implicitly that random forcing is uniformly $\sigma$-$\mathsf{FAM}$-linked.        
        More generally, any measure algebra of Maharam type $\mu$ is uniformly $\mu$-$\mathsf{FAM}$-linked~\cite{MUR23}.

        \item\label{exa:b2:3} The creature ccc forcing from~\cite{HSh}  adding eventually different reals is (uniformly) $\sigma$-$\mathsf{FAM}$-linked. This is proved in~\cite{KST},      witmore general setting in~\cite{M24Anatomy}.
    \end{enumerate}
\end{example}

The author with Mej\'ia proved that fam-limits below help to control $\non(\Ewf)$. Concretely, they proved:

\begin{lemma}[{\cite{CMU}}]\label{lem:d6}
$\sigma$-$\mathsf{FAM}$-linked posets  are  $\Ce$-good.
\end{lemma}

The following results answered our question.

\begin{lemma}[{\cite{BS1992}, see also~\cite{Car4E}}]
Assume $\aleph_1\leq\kappa\leq\lambda=\lambda^{\aleph_0}$ with $\kappa$ regular and assume that $b\in\baire$ satisfies $\sum_{k<\omega}\frac{1}{b(k)}<\infty$. Let $\Bor_{\pi}$ be a FS iteration of random forcing of length $\pi=\lambda\kappa$. Then, in $V^{\Bor_{\pi}}$, 
$\Lc^*\eqT\baire\eqT \Crm_{\Ewf}\eqT\Cbf_{[\lambda]^{<{\aleph_1}}}$ and $\Crm_{\Nwf}^\perp\eqT\Mbf\eqT\kappa$. 
\end{lemma}
\begin{proof}
Since $\Bor$ adds random reals, these are $\Crm_\Nwf$-unbounded reals, which are precisely the $\Cn^\perp$-dominating reals. So by~\autoref{lem:c5} $\Bor_\pi$ forces $\Crm_{\Nwf}^\perp\eqT\Mbf\eqT\lambda_4$ because $\cf(\lambda)=\lambda_4$.

Notice that $\Bor$ is $\sigma$-uf-linked and $\sigma$-$\mathsf{FAM}$-linked (see~\autoref{exa:b3}~\ref{exa:b3:3} and \autoref{exa:b2}~\ref{exa:b2:1}, respectively), so by~\autoref{lem:d11} and \autoref{lem:d6} $\Bor$ is $\baire$-good and $\Ce$-good, respectively.
Thus, $\Bor$ is $\Lc*$-good by~\autoref{exa:c2}~\ref{exa:c2:4}. Hence, by~\autoref{lem:c3}, $\Bor_\pi$ forces $\Cbf_{[\lambda _5]^{<{\aleph_1}}}\leqT\Lc^*$, $\Cbf_{[\lambda _5]^{<{\aleph_1}}}\leqT\baire$, $\Cbf_{[\lambda _5]^{<{\aleph_1}}}\leqT\Crm_{\Ewf}$.

On the other hand, clearly $\Lc^*\leqT\Cbf_{[\lambda _5]^{<{\aleph_1}}}$, $\baire\leqT\Cbf_{[\lambda _5]^{<{\aleph_1}}}$, $\Crm_{\Ewf}\leqT\Cbf_{[\lambda _5]^{<{\aleph_1}}}$ are forced. Consequently, $\Bor_\pi$ forces $\Lc^*\eqT\baire\eqT \Crm_{\Ewf}\eqT\Cbf_{[\lambda]^{<{\aleph_1}}}$.
\end{proof}

\section{Open problems}

We know the consistency $\bfrak_b^\eq>\bfrak$, but the following is not known:

\begin{problem}\label{pro:eo}
Is $\bfrak_b^\eq<\bfrak$ consistently. Dually, Is $\dfrak<\dfrak_b^\eq$ consistent?
\end{problem}

Notice that for $b\in\omega^\omega$, $\balc_{b,1}\leq \non(\Mwf)$ and $\cov(\Mwf)\leq \dalc_{b,1}$. On the other hand, after a FS (finite support) iteration of uncountable cofinality lentgh of ccc non-trivial posets, $\non(\Mwf) \leq \cov(\Mwf)$, which implies by \autoref{dRbup} that $\bfrak \leq \bfrak_b^\eq$ and $\dfrak_b^\eq\leq \dfrak$. Hence, FS iterations cannot solve~\autoref{pro:eo}.

Despite the fact that $\bfrak_b^\eq\leq\non(\Ewf)$ (\autoref{lem:b4}~\ref{lem:b4:1}), we do not know the following:

\begin{problem}
 Is $\bfrak_b^\eq<\non(\Ewf)$ consistent for any (some) $b$?  
\end{problem}

Brendle~\cite{Brppts} (see~also~\cite[Lem.~2.6]{CaraboutE}) proved the consistency of $\non(\Ewf)>\dfrak$, so we ask:

\begin{problem}
Is $\bfrak_b^\eq>\dfrak$ consistent for any (some) $b$?    
\end{problem}

In relation to $\bfrak_b^\eq$ and $\non(\Ewf)$ when $\sum_{k<\omega}\frac{1}{b(k)} = \infty$, we do not know the following:

\begin{problem}\label{pro:e2}
Are each of the following statements consistent with~$\thzfc$?
\begin{enumerate}[label=\rm(\arabic*)]
    \item $\non(\Ewf)<\bfrak_b^\eq$ for any (some) $b$. Dually, $\dfrak_b^\eq<\cov(\Ewf)$ for any (some) $b$.

    \item $\bfrak_b^\eq<\non(\Ewf)$ for any (some) $b$. Dually, $\cov(\Ewf)<\dfrak_b^\eq$ for any (some) $b$.
\end{enumerate} 
\end{problem}

Recently, Yamazoe used uf-limits on intervals (introduced by Mej\'ia~\cite{M24Anatomy}) along FS iterations to construct a~poset to force
\begin{multline*}\label{cicmaxnone}
 \aleph_1<\add(\Nwf)<\cov(\Nwf)<\bfrak<\non(\Ewf)<\non(\Mwf)<\\
 <\cov(\Mwf)<\dfrak<\non(\Nwf)<\cof(\Nwf).
\end{multline*}
The above model can be modified to get the following:
\begin{multline*}
 \aleph_1<\add(\Nwf)<\cov(\Nwf)<\bfrak<\bfrak_b^\eq=\non(\Ewf)<\non(\Mwf)<\\
 <\cov(\Mwf)<\dfrak<\non(\Nwf)<\cof(\Nwf).
\end{multline*}
The point is that $\bfrak_b^\eq\leq\non(\Ewf)$ when $\sum_{k<\omega}\frac{1}{b(k)}<\infty$ and the forcing that increases $\bfrak_b^\eq$ has uniformly $\sigma$-uf-$\lim$-linked (\autoref{lem:b11}). So we ask:

\begin{problem}\label{pro:e1}
\begin{enumerate}[label=\rm(\arabic*)]
    \item\label{pro:e1:1} Is it consistent $\thzfc$ with 
\begin{multline*}
 \aleph_1<\add(\Nwf)<\cov(\Nwf)<\bfrak<\bfrak_b^\eq<\non(\Ewf)<\non(\Mwf)<\\
 <\cov(\Mwf)<\dfrak<\non(\Nwf)<\cof(\Nwf).
\end{multline*}

    \item\label{pro:e1:2} Is it consistent $\thzfc$ with 
\begin{multline*}
 \aleph_1<\add(\Nwf)<\cov(\Nwf)<\bfrak_b^\eq<\bfrak<\non(\Ewf)<\non(\Mwf)<\\
 <\cov(\Mwf)<\dfrak<\non(\Nwf)<\cof(\Nwf).
\end{multline*}
\end{enumerate}
\end{problem}

Notice that FS iterations cannot solve~\ref{pro:e1:2} of~\autoref{pro:e1}. The following could be solved with 
positive answer of~\autoref{pro:e2}:

\begin{problem}
    \begin{enumerate}[label=\rm(\arabic*)]
        \item Is it consistent $\thzfc$ with 
\begin{multline*}
 \aleph_1<\add(\Nwf)<\cov(\Nwf)<\bfrak<\non(\Ewf)<\bfrak_b^\eq<\non(\Mwf)<\\
 <\cov(\Mwf)<\dfrak<\non(\Nwf)<\cof(\Nwf).
\end{multline*} 

    \item Is it consistent $\thzfc$ with 
\begin{multline*}
 \aleph_1<\add(\Nwf)<\cov(\Nwf)<\bfrak_b^\eq<\bfrak<\non(\Ewf)<\non(\Mwf)<\\
 <\cov(\Mwf)<\dfrak<\non(\Nwf)<\cof(\Nwf).
\end{multline*} 
    \end{enumerate}
\end{problem}

\subsection*{Acknowledgments}

This survey has been developed specifically for the proceedings of the RIMS Set Theory Workshop 2024 \textit{Recent Developments in Axiomatic Set Theory}, held at Kyoto University RIMS. The author expresses gratitude to Professor Masahiro Shioya from University of Tsukuba for letting him participate with a talk at the Workshop and submit a paper to this proceedings.

The author would like to thank the Israel Science Foundation for partially supporting this work by grant 2320/23 (2023-2027).

{\small
\bibliography{left}

\begin{thebibliography}{CMRM24}

\bibitem[BCM21]{BCM}
J\"{o}rg Brendle, Miguel~A. Cardona, and Diego~A. Mej\'{\i}a.
\newblock Filter-linkedness and its effect on preservation of cardinal characteristics.
\newblock {\em Ann. Pure Appl. Logic}, 172(1):Paper No. 102856, 30, 2021.

\bibitem[BCM24]{BCM2}
J\"{o}rg Brendle, Miguel~A. Cardona, and Diego~A. Mej\'ia.
\newblock Separating cardinal characteristics of the strong measure zero ideal.
\newblock {\em J. Math. Log.}, 2024.
\newblock To appear, \href{https://arxiv.org/abs/2309.01931}{arXiv:2309.01931}.

\bibitem[BJ95]{BJ}
Tomek Bartoszy\'nski and Haim Judah.
\newblock {\em {Set Theory: On the Structure of the Real Line}}.
\newblock A K Peters, Wellesley, Massachusetts, 1995.

\bibitem[BJ94]{bartJudah}
Tomek Bartoszy\'{n}ski and Haim Judah.
\newblock Borel images of sets of reals.
\newblock {\em Real Anal. Exchange}, 20(2):536--558, 1994/95\red{94}.

\bibitem[BJS93]{BWS}
Tomek Bartoszy\'{n}ski, Winfried Just, and Marion Scheepers.
\newblock Covering {G}ames and the {B}anach-{M}azur {G}ame: {$K$}-tactics.
\newblock {\em Canad. J. Math.}, 45(5):897--929, 1993.

\bibitem[Bla10]{blass}
Andreas Blass.
\newblock Combinatorial cardinal characteristics of the continuum.
\newblock In {\em \-{} Handbook of set theory. {V}ols. 1, 2, 3}, pages 395--489. Springer, Dordrecht, 2010.

\bibitem[Bre91]{Br}
J{\"o}rg Brendle.
\newblock Larger cardinals in {C}icho\'n's diagram.
\newblock {\em J. Symbolic Logic}, 56(3):795--810, 1991.

\bibitem[Bre99]{Brppts}
J\"{o}rg Brendle.
\newblock Between {$P$}-points and nowhere dense ultrafilters.
\newblock {\em Israel J. Math.}, 113:205--230, 1999.

\bibitem[BS92]{BS1992}
Tomek Bartoszy\'{n}ski and Saharon Shelah.
\newblock Closed measure zero sets.
\newblock {\em Ann. Pure Appl. Logic}, 58(2):93--110, 1992.

\bibitem[Car23]{Car4E}
Miguel~A. Cardona.
\newblock A friendly iteration forcing that the four cardinal characteristics of $\mathcal{E}$ can be pairwise different.
\newblock {\em Colloq. Math.}, 173(1):123--157, 2023.

\bibitem[Car24]{CaraboutE}
Miguel~A. Cardona.
\newblock The cardinal characteristics of the ideal generated by the {F}$_{\sigma}$ measure zero subsets of the reals.
\newblock {\em Ky\={o}to Daigaku S\=urikaiseki Kenky\=usho K\=oky\=uroku}, 2290:18--42, 2024.

\bibitem[CM19]{CM}
Miguel~A. Cardona and Diego~A. Mej\'{\i}a.
\newblock On cardinal characteristics of {Y}orioka ideals.
\newblock {\em MLQ}, 65(2):170--199, 2019.

\bibitem[CM22]{CM22}
Miguel~A. Cardona and Diego~A. Mej\'{\i}a.
\newblock Forcing constellations of {C}icho\'n's diagram by using the {T}ukey order.
\newblock {\em Ky\={o}to Daigaku S\=urikaiseki Kenky\=usho K\=oky\=uroku}, 2213:14--47, 2022.

\bibitem[CM23]{CMlocalc}
Miguel~A. Cardona and Diego~A. Mej\'{\i}a.
\newblock Localization and anti-localization cardinals.
\newblock {\em Ky\={o}to Daigaku S\=urikaiseki Kenky\=usho K\=oky\=uroku}, 2261:47--77, 2023.

\bibitem[CM25]{CarMej23}
Miguel~A. Cardona and Diego~A. Mej\'ia.
\newblock More about the cofinality and the covering of the ideal of strong measure zero sets.
\newblock {\em Ann. Pure Appl. Logic}, 176(4):Paper No. 103537, 31, 2025.

\bibitem[CMRM24]{CMR2}
Miguel~A. Cardona, Diego~A. Mej\'{\i}a, and Ismael~E. Rivera-Madrid.
\newblock Uniformity numbers of the null-additive and meager-additive ideals.
\newblock Preprint, \href{https://arxiv.org/abs/2401.15364}{arXiv:2401.15364}, 2024.

\bibitem[CMUZ24]{CMU}
Miguel~A. Cardona, Diego~A. Mejía, and Andrés~F. Uribe-Zapata.
\newblock A general theory of iterated forcing using finitely additive measures.
\newblock Preprint, \href{https://arxiv.org/abs/2406.09978}{arXiv:2406.09978}, 2024.

\bibitem[GMS16]{GMS}
Martin Goldstern, Diego~Alejandro Mej{\'{\i}}a, and Saharon Shelah.
\newblock The left side of {C}icho\'n's diagram.
\newblock {\em Proc. Amer. Math. Soc.}, 144(9):4025--4042, 2016.

\bibitem[HS16]{HSh}
Haim Horowitz and Saharon Shelah.
\newblock Saccharinity with ccc.
\newblock Preprint, \href{https://arxiv.org/abs/1610.02706}{arXiv:1610.02706}, 2016.

\bibitem[JS90]{JS}
Haim Judah and Saharon Shelah.
\newblock The {K}unen-{M}iller chart ({L}ebesgue measure, the {B}aire property, {L}aver reals and preservation theorems for forcing).
\newblock {\em J. Symbolic Logic}, 55(3):909--927, 1990.

\bibitem[Kam89]{Ka}
Anastasis Kamburelis.
\newblock Iterations of {B}oolean algebras with measure.
\newblock {\em Arch. Math. Logic}, 29(1):21--28, 1989.

\bibitem[Kec95]{Ke2}
Alexander~S. Kechris.
\newblock {\em Classical descriptive set theory}, volume 156 of {\em Graduate Texts in Mathematics}.
\newblock Springer-Verlag, New York, 1995.

\bibitem[Kra02]{Kra}
Jan Kraszewski.
\newblock Transitive properties of ideal.
\newblock \url{http://www.math.uni.wroc.pl/~kraszew/sources/papers/trans7.pdf}, 2002.

\bibitem[KST19]{KST}
Jakob Kellner, Saharon Shelah, and Anda~R. T\u{a}nasie.
\newblock Another ordering of the ten cardinal characteristics in {C}icho\'{n}'s diagram.
\newblock {\em Comment. Math. Univ. Carolin.}, 60(1):61--95, 2019.

\bibitem[Mej19]{mejiavert}
Diego~A. Mej{\'{i}}a.
\newblock Matrix iterations with vertical support restrictions.
\newblock In Byunghan Kim, J\"{o}rg Brendle, Gyesik Lee, Fenrong Liu, R~Ramanujam, Shashi~M Srivastava, Akito Tsuboi, and Liang Yu, editors, {\em Proceedings of the 14th and 15th Asian Logic Conferences}, pages 213--248. World Sci. Publ., 2019.

\bibitem[Mej24]{M24Anatomy}
Diego~A. Mejía.
\newblock Anatomy of {$\tilde{\mathbb{E}}$}.
\newblock {\em Ky\={o}to Daigaku S\=urikaiseki Kenky\=usho K\=oky\=uroku}, 2290:43--61, 2024.

\bibitem[Mil81]{Miller}
Arnold~W. Miller.
\newblock Some properties of measure and category.
\newblock {\em Trans. Amer. Math. Soc.}, 266(1):93--114, 1981.

\bibitem[MUZ24]{MUR23}
Diego~A. Mej\'ia and Andr\'es~F. Uribe-Zapata.
\newblock The measure algebra adding $\theta$-many random reals is $\theta$-$\mathrm{FAM}$-linked.
\newblock {\em Topology Appl.}, 2024.
\newblock To appear, \href{https://arxiv.org/abs/2312.13443}{ arXiv:2312.13443}.

\bibitem[Paw85]{paw85}
Janusz Pawlikowski.
\newblock Powers of transitive bases of measure and category.
\newblock {\em Proc. Amer. Math. Soc.}, 93(4):719--729, 1985.

\bibitem[She95]{shmn}
Saharon Shelah.
\newblock Every null-additive set is meager-additive.
\newblock {\em Israel J. Math.}, 89(1-3):357--376, 1995.

\bibitem[She00]{ShCov}
Saharon Shelah.
\newblock Covering of the null ideal may have countable cofinality.
\newblock {\em Fund. Math.}, 166(1-2):109--136, 2000.

\bibitem[Tal80]{Tal98}
Michel Talagrand.
\newblock Compacts de fonctions mesurables et filtres non mesurables.
\newblock {\em Studia Math.}, 67(1):13--43, 1980.

\bibitem[Zin19]{zin19}
Ond{\v{r}}ej Zindulka.
\newblock Strong measure zero and meager-additive sets through the prism of fractal measures.
\newblock {\em Comment. Math. Univ. Carolin.}, 60(1):131--155, 2019.

\end{thebibliography}
\bibliographystyle{alpha}
}

\end{document}